\documentclass[review,hidelinks,onefignum,onetabnum]{siamart220329}

\usepackage{amssymb}
\usepackage{url}
\usepackage{xcolor}

\usepackage{hyperref}
\usepackage{cleveref}
\usepackage{enumerate}
\usepackage[shortlabels]{enumitem}
\usepackage{pbox}
\usepackage{tcolorbox}
\usepackage{bbm}
\usepackage{bm}
\usepackage{tikz}
\usetikzlibrary{decorations.pathreplacing}
\usetikzlibrary{calc}

\DeclareMathOperator*{\esssup}{ess\,sup}

\DeclareMathOperator{\tr}{tr}
\DeclareMathOperator{\diam}{diam}
\DeclareMathOperator{\area}{area}
\DeclareMathOperator{\supp}{supp}
\crefname{prop}{Property}{Properties}

\usepackage{lipsum}
\usepackage{amsfonts}
\usepackage{graphicx}
\usepackage{epstopdf}
\usepackage{algorithmic}
\ifpdf
  \DeclareGraphicsExtensions{.eps,.pdf,.png,.jpg}
\else
  \DeclareGraphicsExtensions{.eps}
\fi

\newsiamremark{remark}{Remark}

\headers{Kernel Smoothing Operators on Thick Open Domains}{Dimitrios  Giannakis, and Mohammad Javad Latifi Jebelli}

\title{Kernel Smoothing Operators on Thick Open Domains\thanks{Submitted to the editors DATE.
\funding{The authors acknowledge support from ONR MURI N00014-19-1-242.}}}

\author{Dimitrios  Giannakis\thanks{Department of Mathematics, Dartmouth College, Hanover, NH
  (\email{dimitrios.giannakis@dartmouth.edu}).}
\and Mohammad Javad Latifi Jebelli\thanks{Department of Mathematics, Dartmouth College, Hanover, NH
  (\email{mohammad.javad.latifi.jebelli@dartmouth.edu}).}}

\ifpdf
\hypersetup{
  pdftitle={Kernel Smoothing Operators on Thick Open Domains},
  pdfauthor={D. Giannakis, M. J. Latifi Jebelli }
}
\fi

\begin{document}
\nolinenumbers

\maketitle

\begin{abstract}
    We define the notion of a thick open set $\Omega$ in a Euclidean space and show that a local Hardy-Littlewood inequality holds in $L^p(\Omega)$, $p \in (1, \infty]$. We then establish pointwise and $L^p(\Omega)$ convergence for families of convolution operators with a Markov normalization on $\Omega$. We demonstrate application of such smoothing operators to piecewise-continuous density, velocity, and stress fields from discrete element models of sea ice dynamics.
\end{abstract}

% REQUIRED
\begin{keywords}
Convolution operators, kernel integral operators, Hardy--Littlewood inequality, approximation of the identity, discrete element models, sea ice
\end{keywords}

% REQUIRED
\begin{MSCcodes}
42B25  , 62G07  , 44A35
\end{MSCcodes}

\section{Introduction}

Let $h \in L^1(\mathbb{R}^n)$ be a non-negative, radially symmetric function with $\int_{\mathbb R^n} h(x)\, dx = 1$, and let $h_{\epsilon}(x)= \epsilon^{-n} h(x/\epsilon)$ for $\epsilon>0$. The  Hardy--Littlewood maximal operator of convolution type, $M_h$, is defined as
\begin{equation}
    M_h f(x) = \sup_{\epsilon>0} \lvert h_{\epsilon}*f(x) \rvert,
    \label{eq:maximal_basic}
\end{equation}
where $f\in L^p(\mathbb{R}^n)$ with $p>1$. The Hardy--Littlewood maximal inequality,
$$
\| M_h f \|_{L^p(\mathbb R^n)} \leq A \| f \|_{L^p(\mathbb R^n)}, \quad A > 0,
$$
can be employed to prove pointwise and $L^p$ convergence of $h_{\epsilon}*f$ to $f$ as $\epsilon \rightarrow 0$. Results of this type find applications in many areas, including function approximation, analysis of partial differential equations (PDEs), and harmonic analysis.

The theory of Hardy-Littlewood maximal functions has been studied and generalized intensively over the past decades; see e.g., Stein \cite{Stein93}. Examples include maximal operators on the Sobolev spaces $W^{1,p}(\mathbb R^n)$, where associated boundedness results are employed in PDE regularity theory \cite{Kinnunen1997, Kinnunen2015,Lerner2008,Shi2023, Cristian2021, Carneiro2013, Carneiro2019}. A more general definition of the maximal operator in the context of a metric measure space $(X,d,\mu)$ is given by
$$
M f(x) = \sup_{\epsilon >0} \frac{1}{\mu(B_{\epsilon}(x))} \int_{B_{\epsilon}(x)} f\, d\mu
$$
where $B_{\epsilon}(x) = \{ y\in X: d(x,y)<\epsilon\}$. For recent developments in this area, see \cite{Weigt2022,Gibara2023,Liu2022, Stempak2015, Tao2014}. A further generalization of the theory is to consider fractional maximal functions,
$$
M^{\alpha} f(x) = \sup_{\epsilon >0} \frac{\epsilon^{\alpha}}{\mu(B_{\epsilon}(x))} \int_{B_{\epsilon}(x)} f\, d\mu, \quad \alpha > 0,
$$
which also have applications in analysis and PDE; e.g., \cite{Beltran2018}. In the context of weighted $L^p$ spaces, a fundamental question is to identify weight functions $w(x)$ such that the maximal operator is bounded on $L^p(\mathbb{R}^n, w(x)\,dx )$ \cite{Viola2014}.

Another variant of the basic definition~\eqref{eq:maximal_basic} is to consider so-called local maximal operators of the form
\[
    M_h f(x) = \sup_{0< \epsilon< \epsilon_0} \lvert h_{\epsilon}*f(x) \rvert
\]
for some $\epsilon_0>0$. Here, the term ``local'' refers to the fact that the supremum is taken locally over $\epsilon \in (0,\epsilon_0)$. See \cite{Kinnunen1998, Viola2014, Di2022, Pan2021, Ramseyer2023} for further details on local maximal functions.

\subsection{Our contributions}

In this paper, we consider an open domain $\Omega \subset \mathbb{R}^n$ as a metric measure space with the metric and measure structures inherited from those of $\mathbb{R}^n$. In this case, the local maximal operator becomes
\begin{equation}
    \tilde M f(x) = \sup_{0 < \epsilon < \epsilon_0} \frac{1}{\mu(\Omega \cap B_{\epsilon}(x))} \int_{\Omega \cap B_{\epsilon}(x)} f\, d\mu.
    \label{eq:local_maximal_basic}
\end{equation}
Here, as a generalization of \eqref{eq:local_maximal_basic}, we consider convolution-type local maximal operators of the form
$$
\tilde M_h f(x) = \sup_{0 < \epsilon < \epsilon_0 } \left|\frac{h_{\epsilon}*f(x)}{h_{\epsilon}*1_{\Omega}(x)} \right|,
$$
with $1_{\Omega}$ being the characteristic function of the set $\Omega$. Building upon a ``thickness'' condition for $\Omega$, which is a related to a local doubling property of metric measure spaces (but is independent from the global notion of doubling), we establish the counterpart of the Hardy--Littlewood maximal inequality for $\tilde M_h$. In our main result, \cref{thm:main}, we use this inequality to prove convergence of $P_{\epsilon}f$ to $f$ pointwise and in $L^p(\Omega)$ norm, where $P_\epsilon$ denotes the normalized smoothing operator given by
$$
P_{\epsilon} f = \frac{h_{\epsilon}*f}{h_{\epsilon}*1_{\Omega}}.
$$
Furthermore, we establish $L^p(\Omega)$ norm convergence results for variants of $P_\epsilon$ associated with bistochastic  kernels using ideas from \cite{CoifmanHirn13,DasEtAl21} (see \cref{thm:bistoch}). A benefit of smoothing with $P_\epsilon$ rather than the raw convolution operator $f \mapsto h_\epsilon *f$ is that the former preserves constant functions on $\Omega$, whereas the latter does not. In the bistochastic case, $P_\epsilon$ additionally preserves integrals of functions; i.e., smoothing is mass-preserving. Aspects of numerical approximation of these operators using triangulation and numerical quadrature over polygonal domains are also studied in this paper.

As an application of our approach, we perform smoothing of piecewise-continuous density, velocity, and stress fields generated by discrete element models (DEMs) of sea ice dynamics over geometrically complex domains. In the following subsection, we outline elements of sea ice dynamics and its representation within DEMs that motivate this application.

\subsection{Motivation: Sea Ice Dynamics}

Sea ice is a critical component of the Earth's climate system. Existing at the interface between the atmosphere and the ocean, it regulates the turbulent transfer of heat, moisture, and momentum between the two media while also drastically affecting radiative heat fluxes via the ice-albedo effect. As a physical medium, sea ice exhibits a highly multiscale nature \cite{Feltham08}. In  regions of relatively low concentration, sea ice cover consists of ice floes (discrete floating chunks of ice) of size $O(\text{10 m})$ to $O(\text{100 km})$ that move under atmospheric and oceanic drag, undergoing a variety of dynamic and thermodynamic processes such as collisions, fracture, ridging, melt and growth. On larger scales, $O(\text{100 km})$ to $O(\text{1000 km})$, the emergent macro-scale behavior is that of a fluid with a highly nonlinear, complex rheology that exhibits aspects of viscous/creep flow for low stress and plastic deformation at sufficiently high stress \cite{Hibler79,HunkeDukowicz97,DansereauEtAl16}. This complex rheology manifests in the formation of intricate networks of narrow failure zones, known as linear kinematic features (LKFs), that can span several hundred kilometers in length while being as narrow as a few tenths of meters \cite{Kwok01}.

Most current-generation sea ice models treat sea ice as a continuum, using PDE solvers to solve the mass and momentum balance equations along with equations for thermodynamics. These models have been reasonably successful in simulating large-scale aspects of sea ice dynamics, such as thickness, concentration, and velocity \cite{KreyscherEtAl00}. With the increase of computational capabilities over the years, and aided by the development of efficient solvers \cite{ShihEtAl23}, continuum models have also been able to successfully simulate aspects of LKFs \cite{HutchingsEtAl05}, even down to kilometer-scale resolution \cite{Zhang21}. Yet, it is widely accepted that at lengthscales approaching the floe scale, and particularly in regions of low ice concentration, the continuum hypothesis eventually breaks down \cite{CoonEtAl07,Herman22}.

Discrete element models (DEMs) offer a alternative approach to sea ice modeling that is well-suited to overcome some of the known biases and limitations of continuum models \cite{BlockleyEtAl20}. Unlike continuum models, DEMs treat sea ice as a granular medium consisting of Lagrangian objects that represent ice floes. Early efforts to discrete element modeling of sea ice date back at least to the mid 1990s \cite{Hopkins96}. Since then, a variety of techniques have been developed that represent ice floes as disks undergoing collisions \cite{DamsgaardEtAl18}, collections of bonded grains of fixed \cite{Herman16,WestEtAl22} or arbitrary \cite{RabatelEtAl15,MoncadaEtAl23} shapes, and interacting polygons \cite{ManucharyanMontemuro22}, among other approaches. Promising results obtained with sea ice DEMs include simulation of arching in the Nares Strait \cite{WestEtAl22}, reproduction of realistic floe-size distributions in winter (high ice concentration) configurations \cite{ManucharyanMontemuro22}, and modeling of quasi-three-dimensional processes such as wave-induced fracture \cite{Herman17}. Besides geophysical applications, DEMs of sea ice have also been employed in engineering applications such as simulation of ice--structure interactions \cite{TukhuriPolojarvi18}.

DEMs present unique challenges when it comes to postprocessing and integration of these models within larger weather and climate models. First, sea ice exhibits highly nonlinear dynamical coupling with the ocean and atmosphere, whose properties influence considerably the nature of upper-ocean turbulence and heat and momentum fluxes between the ocean and atmosphere \cite{GuptaThompson22,ManucharyanThompson22,ZhengMing23}. Given that most atmospheric and oceanic models are based on Eulerian representations of the respective dynamic and thermodynamic fields, the Lagrangian/granular nature of sea ice DEMs complicates this coupling and imposition of boundary conditions, particularly with regards to representation of sharp edges and discontinuities in DEM output. Similar issues are faced by data assimilation systems that have to contend with Lagrangian observations of sea ice floes in order to estimate the state of the underlying ocean circulation \cite{ChenEtAl22b}. Collectively, these challenges motivate the development of smoothing schemes that can produce continuous approximations of DEM variables with well-characterized convergence properties. The ability to extract Eulerian/continuum information from DEM output can also aid the development of improved macro-scale rheologies for sea ice, which would in turn lead to improvement of large-scale continuum sea ice models.

Here, as an application of our kernel smoothing framework, we analyze data from a simulation of  densely packed ice floes under compression and a Nares Strait simulation conducted using the recently developed SubZero DEM \cite{ManucharyanMontemuro22}. We focus on smooth approximations of the mass density, ice velocity, and stress tensor fields obtained by the Markovian integral operators $P_\epsilon$. Applying these operators to piecewise-continuous fields from DEM snapshots yields continuous fields whose smoothness is controlled by the choice of kernel. The compression experiments capture the ice velocity and stress patterns associated with the formation of idealized LKFs in the simulation. Meanwhile, the Nares Strait experiments reveal buildup and release of stress, and associated intermittent ice motion, induced by jamming and fracture of ice floes passing through the Nares Strait.

\subsection{Organization of the paper}

\Cref{sec:mathematical_framework} describes our approach for smoothing with Markov-normalized integral operators on thick domains. \Cref{sec:convergence_analysis} contains an analysis of the pointwise and $L^p$ norm convergence of these operators. \Cref{sec:metric_measure_spaces} examines the relation between thickness and doubling of metric measure spaces.  \Cref{sec:bistoch} studies the case of Markov integral operators with bistochastic kernels. In \cref{sec:numerical_approximation}, we discuss aspects of numerical approximation over polygonal domains. \Cref{sec:experiments} describes our analyses of sea ice DEM simulation data. \Cref{sec:conclusions} contains a summary discussion and our primary conclusions. Proofs of certain auxiliary results are included in \cref{app:proofs}. A Python implementation of the kernel smoothing methods described in this paper can be found in the repository \url{https://github.com/SeaIce-Math/KSPoly}.

\section{Kernel smoothing on thick domains}
\label{sec:mathematical_framework}

Let $\Omega$ be an open subset of $\mathbb{R}^n$. We consider the problem of approximating measurable functions $f: \Omega \to \mathbb R $ using kernel operators conditioned on $\Omega$.

\subsection{Notation}

For a Banach space $E$, $B(E)$ will denote the Banach space of bounded linear operators on $E$, equipped with the operator norm. Given a Lebesgue-measurable set $X \subseteq \mathbb R^n$, $L(X)$ will denote the space of Lebesgue-measurable real-valued functions on $X$, and $L^p(X)$ with $1 \leq p \leq \infty$ will denote the standard $L^p$ spaces on $X$ equipped with the norms $\lVert f\rVert_{L^p(X)} = (\int_x \lvert f(x)\rvert^p\, dx)^{1/p}$ for $ 1 \leq p < \infty$ and $\lVert f\rVert_{L^\infty(X)} = \esssup_{x\in X} \lvert f(x)\rvert$. Moreover, $C(X)$ will be  the space of continuous real-valued functions on $X$, and $C_c(X)$, $C_b(X)$ the subspaces of $C(X)$ consisting of compactly supported and bounded functions, respectively. We equip $C_c(X)$ and $C_b(X)$ with the uniform norm, $\lVert f\rVert_\infty = \sup_{x\in X} \lvert f(x)\rvert$. If $X$ is open, $C^r(X)$ will be the space of $r$-times continuously differentiable functions on that set. Given a subset $X \subseteq \mathbb R^n$, $\bar X$ will denote its closure and $1_X : \mathbb R^n \to \mathbb R$ will be the characteristic function of that set. For $x\in \mathbb R^n$, $T_x : L(\mathbb R^n) \to L(\mathbb R^n)$ will denote the translation operator $T_x f(y) = f(y-x)$. Moreover, $^\checkmark: L(X) \to L(X)$ will be the ``flip'' operator, $f^\checkmark(x) = f(-x)$, and $f_1 * f_2 (x) = \int_{\mathbb R^n} f_1(x-y) f_2(y) \, dy$ the convolution of $f_1,f_2 \in L(\mathbb R)$ (whenever that operation is well-defined). We let $\mathcal F : L^2(\mathbb R^n) \to \mathbb R^n$ denote the unitary Fourier operator on $L^2(\mathbb R^n)$, defined by extension of $\mathcal F f = \int_{\mathbb R^n} e^{-i \omega \cdot x} f(x) \, dx$ for $f \in C_c(\mathbb R^n)$. We shall denote the Euclidean norm of a vector $x \in \mathbb R^n$ as $\lVert x\rVert$. For $r \geq 0$ we define $\hat H^r(\mathbb R^n) \subseteq L^2(\mathbb R^n)$ as the closure of $C_c(\mathbb R^n)$ with respect to the norm $\lVert \hat f\rVert_{H^r(\mathbb R^n)} = (\int_{\mathbb R^n} \lvert \hat f(\omega) \rvert^2 (1 + \lVert \omega\rVert^2)^r  \, d\omega)^{1/2}  $ and $H^r(\mathbb R^n) = \mathcal F^*\hat H^r(\mathbb R^n) $ as the inverse Fourier image of $\hat H^r(\mathbb R^n)$. For $r \in \mathbb N_0$, $H^r(\mathbb R^n)$ is isomorphic to the Sobolev subspace of $L^2(\mathbb R^n)$ of order $r$ defined in terms of weak derivatives.

\subsection{Radial kernels and associated integral operators}

The starting point of our construction is a positive, radially symmetric, normalized function $h \in L^1(\mathbb R^n)$. Specifically, we require:

\begin{enumerate}[label=(K\arabic*)]
    \item \label[prop]{prop:K1} $h(x) \geq 0 $ for every $x \in \mathbb R^n$.
    \item  \label[prop]{prop:K2}$h(Ox) = h(x)$ for every $x \in \mathbb R^n$ and orthogonal matrix $O \in \mathbb R^{n\times n}$.
    \item \label[prop]{prop:K4} $\lVert h\rVert_{L^1(\mathbb R^n)} = 1$.
\end{enumerate}
In some cases, we will additionally require:
\begin{enumerate}[label=(K\arabic*)]
    \setcounter{enumi}{3}
    \item \label[prop]{prop:K3} $h$ is radially decreasing; that is, $h(y) \leq h(x)$ for every $x, y \in \mathbb R^n$ such that $\lVert x\rVert \leq \lVert y \rVert$.
    % \item \label[prop]{prop:KL2} $h$ lies in the restricted space $L^1(\mathbb R^n) \cap L^2(\mathbb R^n)$.
\end{enumerate}
A concrete example of a function having all of \cref{prop:K1,prop:K2,prop:K3,prop:K4} (which we will use in the numerical experiments of \cref{sec:experiments}) is the Gaussian radial basis function (RBF),
\begin{equation}
    h(x) = \frac{1}{(2\pi)^{n/2}} e^{-\lVert x\rVert^2/2}.
    \label{eq:h_gaussian}
\end{equation}

To any $h$ satisfying \cref{prop:K1,prop:K2,prop:K4}, we associate a one-parameter family of functions $h_\epsilon : \mathbb R^n \to \mathbb R$ and corresponding kernels $k_\epsilon : \mathbb R^n \times \mathbb R^n \to \mathbb R$, where
\begin{equation}
    h_\epsilon(x) = \epsilon^{-n} h(x/\epsilon), \quad k_\epsilon(x, y) = h_\epsilon(y-x).
    \label{eq:h_kepsilon}
\end{equation}
In what follows, we will refer to kernels $k_\epsilon$ obtained in this manner as \emph{radial kernels} and the underlying function $h$ as \emph{shape function}. Note that every radial kernel is translation-invariant and symmetric, i.e.,
\begin{displaymath}
    k_\epsilon(x+z, y+z) = k_\epsilon(x, y), \quad k_\epsilon(x,y) = k_\epsilon(y,x),
\end{displaymath}
respectively, for all $x,y,z \in \mathbb R^n$.

%The kernel $k_\epsilon$ is said to be $r$-times continuously differentiable on an open set $\Omega \subseteq \mathbb R^n$ if, when viewed as a function on $\Omega\times\Omega \subseteq \mathbb R^{2n}$, its derivatives $\partial_1^{j_1}\cdots\partial_n^{j_n}\partial_{n+1}^{j_1}\partial_{2n}^{j_n} k_\epsilon$ are continuous functions for every $j_1,\ldots,j_n\in \mathbb N_0$ satisfying $\sum_{i=1}^n j_i \leq r$ (see, e.g., \cite[Definition~4.35]{SteinwartChristmann08}). \dgm{Reminder to check if we actually use this definition.} Moreover, $k_\epsilon$ is said to be smooth if it is $r$-times continuously differentiable for every $r \in \mathbb N$.

Every radial kernel $k_\epsilon$ induces an integral operator $K_\epsilon: L^1(\mathbb R^n) \to L^1(\mathbb R^n)$ that acts on functions by convolution,
\begin{equation}
    K_\epsilon f := \int_{\mathbb R^n} k_\epsilon(\cdot,y) f(y) \, dy \equiv h_\epsilon * f.
    \label{eq:k_op}
\end{equation}
In particular, since $\lVert h_\epsilon * f \rVert_{L^1(\mathbb R^n)} \leq \lVert h_\epsilon \rVert_{L^1(\mathbb R^n)} \lVert f \rVert_{L^1(\mathbb R^n)}$ and $\lVert h_\epsilon\rVert_{L^1(\mathbb R^n)} = \lVert h \rVert_{L^1(\mathbb R^n)} = 1$, $K_\epsilon$ is a contraction on $L^1(\mathbb R^n)$,
\begin{equation}
    \lVert K_\epsilon f\rVert_{L^1(\mathbb R^n)} \leq \lVert f\rVert_{L^1(\mathbb R^n)}.
    \label{eq:k_bound}
\end{equation}
We will continue to use the symbol $K_\epsilon$ to denote integral operators acting on function spaces other than $L^1(\mathbb R^n)$ via the formula~\eqref{eq:k_op}.

Next, we state some basic results on the regularity of functions in the range of $K_\epsilon$.

\begin{lemma}
    \label{lem:conv_contin}
    Set $p \in (1,\infty]$ and $q \in [1,\infty)$ with $\frac{1}{p} + \frac{1}{q} = 1$. Then, for every $h \in L^q(\mathbb R^n)$ and $f \in L^p(\mathbb R^n)$, the function $K_\epsilon f$ is continuous and bounded with $ \lVert K_\epsilon f \rVert_\infty \leq \lVert h \rVert_{L^q(\mathbb R^n)} \lVert f\rVert_{L^p(\mathbb R^n)} $.
\end{lemma}

\begin{proof}
    See \cref{app:proofs}.
\end{proof}

It follows by \cref{lem:conv_contin} that $K_\epsilon$ maps bounded measurable functions to continuous functions. Moreover, it can be shown that if $\Omega \subset \mathbb{R}^n$ is open and $h$ is smooth and compactly supported, $K_\epsilon$ maps locally integrable functions to functions in $C^{\infty}(\Omega\setminus \Omega_{\epsilon})$, where $\Omega_{\epsilon}=\{x\in \Omega: B_{\epsilon}(x)\not\subset\Omega \}$ (see, e.g., \cite[Appendix~C.4]{Evans2022}).
% (mollifier)
\begin{lemma}
    \label{lem:sobolev}
    For every $h \in H^r(\mathbb R^n)$, $r \in \mathbb N_0$, and $f \in L^2(\mathbb R^n)$ the function $K_\epsilon f$ lies in $C^r(\mathbb R^n)$.
\end{lemma}

\begin{proof}
    See \cref{app:proofs}.
\end{proof}

%If a shape function $h$ belongs to the Sobolev space $W^k(\mathbb{R}^n)$ then $K_{\epsilon}f \in W^k(\mathbb{R}^n)$ for any $f\in L^2(\mathbb{R}^n) = W^0(\mathbb{R}^n)  $. To see why, note that $h \in W^k(\mathbb{R}^n)$ if and only if $(1-|\zeta|^2)^{k/2} \hat{h} \in L^2(\mathbb{R}^n)$, in the fourier domain. The integrability of $h$ implies $\hat{h} \in L^{\infty}(\mathbb{R}^n)$ and together with $(1-|\zeta|^2)^{k/2} \hat{h} \in L^2(\mathbb{R}^n)$ one can deduce that $(1-|\zeta|^2)^{k/2} \hat{h}$ is in the intersection of $L^2(\mathbb{R}^n)$ and $L^{\infty}(\mathbb{R}^n)$. As a result,
%$$
%(1-|\zeta|^2)^{k/2} \widehat{h*f} = (1-|\zeta|^2)^{k/2} \hat{h} \hat{f}
%$$
%belongs to $L^2(\mathbb{R}^n)$, therefore $h*f \in W^k(\mathbb{R}^n)$. \\
%%In particular, $K_\epsilon$ maps the floe-local functions $f_{t,\ell}$ from~\eqref{eq:f_ext} (which are continuous and compactly supported, and thus bounded, by construction) to continuous functions.

%If, in addition, $k_\epsilon$ is $r$-times continuously differentiable, then $K_\epsilon f$ lies in $C^r(\mathbb R^n)$ whenever $f \in L^\infty(\mathbb R^n)$. \dg{Since we are no longer dealing with compact sets, I don't think that $k_\epsilon \in C^r$ is enough to deduce that $K_\epsilon f \in C^r$. For that, we might need that $h$ lies in the Sobolev space $W^{r,1}(\mathbb R^n)$, but I'm not sure.}

\subsection{Smoothing with integral operators}
\label{sec:integral_operators}

Given a function $f \in L^1(\Omega)$, we view $K_\epsilon f$ from~\eqref{eq:k_op} as an approximation of $f$ with regularity controlled by the kernel $k_\epsilon$. Intuitively, as $\epsilon$ decreases, the locality of $k_\epsilon$ increases, and we expect $K_\epsilon f$ to converge to $f$ in some sense. To illustrate this convergence with a concrete example, we recall the Lebesgue differentiation theorem. Let $B_\epsilon(x)$ be the open ball of radius $\epsilon$ centered at $x \in \mathbb R^n$ and $\lvert B_\epsilon(x)\rvert$ its Lebesgue measure. Let also $B \equiv B_1(0)$ be the unit ball centered at the origin. The Lebesgue differentiation theorem states that for every measurable function $f: \mathbb R^n \to \mathbb R$ the functions defined as
$$
    \tilde g_\epsilon(x) = \frac{1}{|B_{\epsilon}(x)|} \int_{B_{\epsilon}(x)} f(y) \,d y
$$
for a.e.~$x \in \mathbb R^n$ converge as $\epsilon\to 0^+$ to $f$ a.e. One can rewrite the above expression in kernel form by defining
\begin{equation}
    \label{eq:tophat}
    l(x) = \frac{1}{\lvert B\rvert} 1_B(x), \quad l_\epsilon(x) = \epsilon^{-n} 1_B(x/\epsilon), \quad \tilde k_\epsilon(x, y) = l_\epsilon(y - x),
\end{equation}
so that
\begin{displaymath}
    \tilde g_\epsilon(x) = \int_{\mathbb R^n} \tilde k_\epsilon(x,y) f(y)\, dy \equiv l_\epsilon * f(x),
\end{displaymath}
for a.e.\ $x \in \mathbb R^n$. As one can readily verify, $l$ satisfies \cref{prop:K1,prop:K2,prop:K3,prop:K4}, and thus we can interpret $\tilde g_\epsilon$ as a special case of~\eqref{eq:k_op} for the shape function set to the ``tophat'' function $1_B$ supported on the unit ball. In other words, the Lebesgue differentiation theorem provides an example of kernel-smoothed functions $\tilde g_\epsilon$ that pointwise-recover the input function $f$.

The  a.e.\ convergence result for $\tilde g_\epsilon$ to $f$ based on the Lebesgue differentiation theorem can be extended to general radial kernels using the theory of Hardy-Littlewood maximal functions. Given a shape function $h \in L^1(\mathbb R^n)$, the local Hardy-Littlewood maximal function associated with $f \in L^p\left(\mathbb{R}^n\right)$ is defined as
\begin{equation}
    M_h f(x)=\sup _{\epsilon > 0}\left|h_\epsilon * f(x)\right|,
    \label{eq:maximal}
\end{equation}
with $h_\epsilon$ defined in~\eqref{eq:h_kepsilon}. The following theorems can be found in \cite[pp.~23, 71]{Stein93}.

\begin{theorem}[Local Hardy-Littlewood maximal inequality]\label{thm:A}
    Suppose that $h \in L^1(\mathbb R^n)$ satisfies \cref{prop:K1,prop:K2,prop:K3,prop:K4}. Then, for every $1<p\leq\infty$ there exists $A>0$ such that for every $f \in L^p\left(\mathbb{R}^n\right)$ the maximal function $M_h f$ satisfies
$$
\left\|M_h f\right\|_{L^p(\mathbb R^n)} \leq A\|f\|_{L^p(\mathbb R^n)}.
$$
\end{theorem}

\begin{theorem}
    \label{thm:B}
    Suppose that $h \in L^1(\mathbb R^n)$ satisfies \cref{prop:K1,prop:K2,prop:K3,prop:K4} and the decay bound $|h(x)| \leq C (1+|x|)^{-n-\delta}$ for some $C,\delta > 0$. Then, for every $f \in L^p\left(\mathbb{R}^n\right)$ with $1<p\leq\infty$, the family of functions $K_\epsilon f = h_\epsilon * f$ lies in $L^p(\mathbb R^n)$ and converges to $f$ a.e.\ as $\epsilon\to 0^+$.
\end{theorem}

By \cref{thm:B}, we see that functions in $L^p(\mathbb R^n)$ with $ 1 < p \leq \infty$ can be pointwise approximated by smoothed functions using general radial kernels. Combining \cref{thm:A,thm:B} yields:

\begin{corollary}
    \label{cor:lp_conv_k}
    With the assumptions and notation of \cref{thm:B}, $K_\epsilon f$ converges to $f$ in $L^p(\mathbb R^n)$ norm for every $p \in (1,\infty)$.
\end{corollary}

\begin{proof}
    The claim follows similarly to Claim~2 of \cref{thm:main}, which is proved in \cref{sec:proof_main}.
\end{proof}

Despite these encouraging results, a shortcoming of smoothing using integral operators such as $K_\epsilon$ that are not adapted to the domain $\Omega$ is that, in general, if $f$ is constant on $\Omega$, $K_\epsilon f$ will not be a constant function on that domain. Moreover, if $f$ and $\tilde g_\epsilon$ are integrable on $\Omega$, the integrals $\int_\Omega f(x) \, dx$ and $\int_\Omega K_\epsilon f(x) \, dx$ will generally not be equal. Preservation of constant functions and integrals are important properties, particularly when dealing with observables of physical systems. In the following subsection, we describe a modified scheme that employs Markov-normalized versions of $K_\epsilon$ on domains satisfying a ``thickness'' condition to ensure preservation of constant functions and integrals, as well as pointwise and $L^p$ norm convergence.

\subsection{Thick domains and Markov normalization}
\label{sec:markov_normalization}

Given a measurable set $\Omega \subseteq \mathbb R^n$, a measurable function $p: \Omega \times \Omega \to \mathbb R$ will be said to be a \emph{Markov kernel} if (i) $p(x,y) \geq 0$; and (ii) $\int_\Omega p(x,s)\,ds =1 $ for all $x, y \in \Omega$. Every Markov kernel has the property that for a constant function $f : \Omega \to \mathbb R$, the function $g: \Omega \to \mathbb R$ with
\begin{equation}
    \label{eq:g_integral}
    g(x) = \int_\Omega p(x,y) f(y) \, dy
\end{equation}
is also constant and equal to $f$. This suggests that smoothing operators associated with Markov kernels are good candidates for addressing the failure of smoothing using un-normalized kernels to preserve constant functions. We now describe a procedure to normalize a family of radial kernels $k_\epsilon : \mathbb R^n \times \mathbb R^n \to \mathbb R$ to a corresponding family of Markov kernels $p_\epsilon : \Omega \times \Omega \to \mathbb R$ such that the formula~\eqref{eq:g_integral} leads to bounded integral operators on $L^p(\Omega)$ that preserve constant functions.

Our construction employs a notion of thickness of the domain $\Omega$, which we now define.

\begin{definition}[Thick set]\label{def:thick}
    A Lebesgue-measurable set $\Omega \subseteq \mathbb{R}^n$ is said to be \emph{thick} if there exists $c, \epsilon_0>0$ such that for all $\epsilon \in (0, \epsilon_0)$ and $x \in \Omega$,
    $$
    \left| B_{\epsilon}(x) \cap \Omega \right| > c \left| B_{\epsilon}(x) \right|.
    $$
\end{definition}

Note that the inequality in \cref{def:thick} can be equivalently expressed as $l_{\epsilon} * 1_{\Omega}(x) > c$, where $l_\epsilon$ is the scaled tophat function from~\eqref{eq:tophat}. By the Lebesgue differentiation theorem, the function
\begin{equation}
    \rho_{\Omega}(x) = \lim_{\epsilon \rightarrow 0} \frac{\left| B_{\epsilon}(x) \cap \Omega \right|}{\left| B_{\epsilon}(x) \right|},
    \label{eq:lebesgue_density}
\end{equation}
known as the \emph{Lebesgue density function} of $\Omega$, is well-defined for a.e.~$x \in \mathbb R^n$.  We can thus interpret thick sets as measurable subsets of $\mathbb R^n$ with approximate Lebesgue densities bounded below. Examples of thick sets include polygons, solid spheres, and other domains commonly used in numerical applications; see \cref{sec:thicklemma}. More generally, for an interior point $x$ of a measurable set $\Omega \subseteq \mathbb R^n $ and for small-enough $\epsilon$ we have $B_{\epsilon}(x) \subset \Omega$---this implies that the limit in~\eqref{eq:lebesgue_density} exists and $\rho_\Omega(x) = 1$ for every interior point. However, for points in the boundary of $\Omega$ the Lebesgue density may not exist.

Given a shape function $h \in L^1(\mathbb R^n)$ and a Borel-measurable set $\Omega\subseteq \mathbb R^n$, we define for each $\epsilon > 0 $ the function $d_\epsilon : \Omega \to \mathbb R$ as
\begin{equation}
   d_\epsilon = K_\epsilon 1_\Omega \equiv h_\epsilon * 1_\Omega.
   \label{eq:degree_fn}
\end{equation}
By \cref{lem:conv_contin}, $d_\epsilon$ is continuous on $\mathbb R^n$ and it is clearly positive. If, in addition, $h$ is smooth and compactly supported, local integrability of $1_\Omega$ implies that $d_{\epsilon}$ lies in $C^\infty(\Omega\setminus \Omega_\epsilon)$.  We will refer to $d_\epsilon$ as a \emph{degree function} by analogy with the notion of degree vectors associated with discrete kernel integral operators on graphs; e.g., \cite{Chung97}.

The following proposition provides sufficient conditions for the strict positivity of $d_\epsilon$ that allow us to employ these functions them for Markov normalization.

\begin{proposition}\label{prop:d_func}
    Let $h \in \mathbb L^1(\mathbb R^n)$ be a shape function satisfying \cref{prop:K1,prop:K2,prop:K4}, $\Omega \subseteq \mathbb R^n$ an open set, and $d_\epsilon$ the corresponding degree functions from~\eqref{eq:degree_fn}.
    \begin{enumerate}
        \item If $\Omega$ is bounded, and in addition $h$ satisfies \cref{prop:K3}, then for every $\epsilon > 0$ there exists $c_\epsilon>0$ such that for every $x \in \bar \Omega$ we have $d_\epsilon(x) \geq c_\epsilon$.
        \item If $\Omega$ is thick, then there exists $c,\epsilon_0 > 0$ such that for every $\epsilon\in (0, \epsilon_0)$ and $x \in \bar \Omega$ we have $d_\epsilon(x) \geq c$.
    \end{enumerate}
\end{proposition}

\begin{proof}
    Claim~2 is proved as \cref{lem:deps_lower_bound} in \cref{sec:thicklemma}.

    To prove Claim~1, after a change of variables, it is enough to consider the case $x=0$. We claim that if $0 \in \bar\Omega$ then
    \begin{equation}\label{eq:hOmega}
        \int_{\Omega} h_{\epsilon}(x)\, dx > 0
    \end{equation}
for any $\epsilon >0$. To prove this, take an arbitrary open ball $B_{\epsilon}$ of radius $\epsilon$ centered at the origin. Since $B_{\epsilon} \cap \Omega$ is open, we can find another ball $B \subset B_{\epsilon} \cap \Omega$. Assume, for contradiction, that $\int_{\Omega} h_{\epsilon}(x)\, dx = 0$.  By \cref{prop:K1}, this implies $\int_{B} h_{\epsilon}(x)\, dx = 0$, and thus $h(x)=0$ for a.e.\ $x \in B$. Now let $A = \{ Ox \in \mathbb{R}^n: O \in O(n), \, x \in B \}$ be the spherical shell generated by $B$. By \cref{prop:K2}, $h$ vanishes almost everywhere on $A$ and \cref{prop:K3} implies that $h(x)=0$ for a.e.\ $x \in \mathbb{R}^n \setminus B_{\epsilon}$. But $\epsilon>0$ was arbitrary, which implies that $h(x)=0$ a.e., contradicting \cref{prop:K4} and proving \eqref{eq:hOmega}.

Using a change of coordinates, \eqref{eq:hOmega} implies that $d_{\epsilon}(x)> 0$ for any $\epsilon > 0$ and $x \in \bar\Omega$. Since $\Omega$ is bounded and $d_{\epsilon}$ is continuous, by compactness of $\bar \Omega$ $d_{\epsilon}$ takes its minimum value at some point $x_0 \in \bar \Omega$. The statement of the proposition follows by setting $c_{\epsilon} = d_{\epsilon}(x_0)$, giving  $d_{\epsilon}(x) \geq c_0 > 0$.
\end{proof}

Let $\Omega \subseteq \mathbb R^n$ and $\epsilon_0 > 0$ be such that $d_\epsilon$ satisfies \cref{prop:d_func} for every $\epsilon \in (0, \epsilon_0)$. Then, for every such $\epsilon$, $p_\epsilon : \Omega \times \Omega \to \mathbb R$ with
\begin{equation}
    p_\epsilon(x,y) = \frac{k_\epsilon(x, y)}{d_\epsilon(x)}
    \label{eq:p_kernel}
\end{equation}
is a well-defined measurable function on $\Omega \times \Omega$. By construction, this function satisfies $p_\epsilon(x, y) \geq 0 $ and $\int_\Omega p_\epsilon(x,s)\, ds = 1$ for every $x,y \in \Omega$, and is therefore a Markov kernel. Note that, unlike $k_\epsilon$, $p_\epsilon$ is in general neither translation-invariant nor symmetric. The procedure to obtain $p_\epsilon$ from $k_\epsilon$ is sometimes referred to as left normalization; e.g., \cite{BerrySauer16}.

\subsection{Smoothing using Markovian kernels}
Similarly to the kernel $k_\epsilon$, the Markov kernel $p_\epsilon$ from~\eqref{eq:p_kernel} has an associated bounded integral operator $P_\epsilon: L^1(\Omega) \to L^1(\Omega)$, defined as
\begin{equation}
    P_\epsilon f = \int_\Omega p_\epsilon(\cdot, y)f(y) \, dy.
    \label{eq:p_op}
\end{equation}
We have $P_\epsilon 1_\Omega = 1_\Omega$ from which it follows that $P_\epsilon$ preserves constant functions on $\Omega$. If $\Omega$ is bounded, we have $L^p(\Omega) \subseteq L^1(\Omega)$ for all $ p \in [1, \infty]$, and thus $P_\epsilon$ is well-defined as a bounded operator from $L^p(\Omega)$ into $L^1(\Omega)$. Since $K_\epsilon$ maps $L^\infty(\mathbb R^n)$ to $C_b(\mathbb R^n)$ and $1/d_\epsilon$ is bounded and continuous (the latter, by \cref{prop:d_func}), we have that $P_\epsilon$ maps functions in $L^\infty(\Omega)$ to continuous bounded functions on $\Omega$.

By the above facts and \cref{lem:sobolev}, we think of $P_\epsilon$ as a smoothing operator on $L^p(\Omega)$ that preserves constant functions. Given a function $f \in L^p(\Omega)$, we will employ $P_\epsilon f$ as an approximation to $f$. The following theorem is our main result that establishes the convergence properties of this approximation under suitable assumptions on the domain and  kernel.

\begin{theorem}
    \label{thm:main}
    Let $\Omega$ be a thick open subset of $\mathbb{R}^n$, and suppose that the shape function $h$ satisfies \cref{prop:K1,prop:K2,prop:K4,prop:K3}. Then, there exists $\epsilon_0 > 0$ such that for every $p \in [1, \infty]$, $\{P_\epsilon\}_{0<\epsilon<\epsilon_0}$ from~\eqref{eq:p_op} is a uniformly bounded family of operators from $L^p(\Omega)$ to itself. If, in addition, $h$ satisfies the decay bound $\lvert h(x) \rvert \leq C (1+ \lVert x\rVert)^{-(n+\delta)}$ for some $C,\delta > 0 $, the following hold.
    \begin{enumerate}
        \item For every $p \in (1, \infty]$ and $f \in L^p(\Omega)$, $P_\epsilon f$ converges to $f$ a.e.
        \item For every  $p \in (1, \infty)$ and $f \in L^p(\Omega)$, $P_\epsilon f$ converges to $f$ in $L^p(\Omega)$ norm.
        \item If $h$ lies in $L^1(\mathbb R^n) \cap L^q(\mathbb R^n)$ with $\frac{1}{p}+\frac{1}{q} = 1$, then for every $p \in (1,\infty]$ and $f \in L^p(\Omega)$  $P_\epsilon f$ lies in $C_b(\Omega)$.
        \item If $\Omega$ is bounded and $h \in L^1(\mathbb R^n) \cap H^r(\mathbb R^n)$, then for every $f \in L^2(\Omega)$ $P_\epsilon f$ lies in $C^r(\Omega)$.
    \end{enumerate}
\end{theorem}

\begin{proof}
    See \cref{sec:proof_main}.
\end{proof}

\section{Convergence analysis}
\label{sec:convergence_analysis}

The focus of this section will be on proving \cref{prop:d_func} and \cref{thm:main}. We begin by deriving some auxiliary results on thick sets (\cref{sec:thicklemma}) used in the proof of \cref{prop:d_func}. We then examine the case of smoothing bounded measurable functions (\cref{sec:boundedf}). We prove \cref{thm:main} in \cref{sec:proof_main} using the results from \cref{sec:thicklemma,sec:boundedf}.

\subsection{Convolution and thick open sets}\label{sec:thicklemma}

We prove a technical lemma that addresses Claim~2 of \cref{prop:d_func}. The goal is to leverage thickness of $\Omega$ to deduce that for a given shape function $h$ and for small-enough $\epsilon > 0$ (as well as in the limit), the degree function $d_\epsilon = h_{\epsilon}*1_{\Omega}$ from \eqref{eq:degree_fn} satisfies  $d_\epsilon(x) \geq c > 0$ for every $x \in \bar\Omega$. Let us first review a few examples and non-examples of thick open sets.

\paragraph*{A non-example} We construct a (non-thick) bounded open subset $ \Omega \subset \mathbb R$ with no strictly positive lower bound for $l_\epsilon * 1_\Omega$. To a map $r: \mathbb{N} \rightarrow[0,1]$ we associate an open set $\Omega =\bigcup_{n=1}^{\infty}\left(a_n, b_n\right),$ where $\frac{1}{n+1} \leq a_n < b_n \leq \frac{1}{n}$ and
\begin{displaymath}
    b_n + a_n = \frac{1}{n}+\frac{1}{n+1},\quad b_n-a_n =r_n\left(\frac{1}{n}-\frac{1}{n+1}\right).
\end{displaymath}
Intuitively, $\left(a_n, b_n\right)$ is an interval centered at the center of $I = \left(\frac{1}{n+1}, \frac{1}{n}\right)$ that covers a fraction $r_n$ of the length (Lebesgue measure) of $I$. For $\varepsilon=\frac{1}{2 n}-\frac{1}{2 n+2}=\frac{1}{2 n(n+1)}$ and the center point $x_n=\frac{1}{2 n}+\frac{1}{2 n+2}$ we have
$$
    l_{\epsilon}*1_{\Omega}(x_n) = \frac{\left|B_{\varepsilon}\left(x_n\right) \cap \Omega\right|}{\left|B_{\varepsilon}\left(x_n\right)\right|}=r_n,
$$
and we conclude that if $r_n \rightarrow 0$ then $\Omega$ is not thick, i.e., there is no lower bound for $l_{\epsilon}*1_{\Omega}$ as $\epsilon \rightarrow0$.

\paragraph*{Triangles} A triangle in $\mathbb R^2$ is thick. It is not hard to see that if $\theta$ is the smallest angle in a given triangle $\Omega$, then for small enough $\epsilon>0$ we have $l_{\epsilon}*1_{\Omega}(x) \geq \frac{\theta}{2\pi}$ for any $x \in \bar{\Omega}$.

\paragraph*{Finite unions of thick sets} Let $\Omega_n$ be a finite collection of thick sets such that for small $\epsilon>0$, $l_{\epsilon}*1_{\Omega_n}(x) \geq c_n$. Then, $\Omega = \bigcup_{n}\Omega_n$ is a thick set with $l_{\epsilon}*1_{\Omega}(x) \geq \min c_n$. In particular, a finite union of triangles is thick, and this includes polygons and more general irregular domain regions formed by finitely many triangles.

\begin{lemma}\label{lem:deps_lower_bound}
  Let $\Omega$ be a thick open subset of $\mathbb{R}^n$ and $h$ a shape function satisfying \cref{prop:K1,prop:K2,prop:K4}. Then, there exists $c,\epsilon_0 > 0$ such that for every $\epsilon \in (0,\epsilon_0)$ and $x \in \bar\Omega$, we have
$$
d_\epsilon(x) \equiv h_{\varepsilon} * 1_{\Omega}{(x)} \geq c.
$$
\end{lemma}

\begin{proof}
    For any $\delta>0$, there exists a simple radial function $\tilde h\geq 0$ satisfying $\lVert h-\tilde h \rVert_{L^1(\mathbb R^n)}<\delta$.
    Such $\tilde h$ can be expressed as a finite sum $\tilde h(x)=\sum_{i=1}^s c_i l_{r_i}(x)$, where $s \in \mathbb N$, $l_r(x)=r^{-n} 1_B(x / r)$, $\sum_{i=1}^s c_i=\int_{\mathbb R^n} \tilde h(x)\, d x=1$, and $c_i \geq 0$. Note that $\left\|l_r\right\|_{L^1(\mathbb R^n)}=1$. Since $\Omega$ is thick, there exist $c,\epsilon_0>0$ such that for every $i \in \{1, \ldots, s\}$, $\epsilon \in (0, \epsilon_0)$, and $x \in \Omega$, we have $l_{\epsilon r_i}* 1_{\Omega}(x)>c$. But the same lower bound, $c$, applies to any convex combination of $l_{\epsilon r_i}* 1_{\Omega}(x)$, and thus
\begin{equation}
    \label{eq:h_ineq1}
    h_{\epsilon}^{\prime} * 1_{\Omega}(x)=\sum_{i=1}^s c_i l_{\epsilon r_i} * 1_{\Omega}{(x)}>c.
\end{equation}
On the other hand, from \cref{lem:conv_contin}, we have
\begin{equation}
    \label{eq:h_ineq2}
    \lVert h_\epsilon * 1_\Omega - h'_\epsilon * 1_\Omega \rVert_\infty \leq \lVert h - h'\rVert_{L^1(\mathbb R^n)} < \delta.
% \sup _{x \in \mathbb{R}^n}\left|h_{\epsilon} * 1_{\Omega}{(x)}-h_{\epsilon}^{\prime} * 1_{\Omega}{(x)}\right| \leqslant\left\|h-h^{\prime}\right\|_1<s
\end{equation}
Combining \eqref{eq:h_ineq1} and~\eqref{eq:h_ineq2}, we conclude that for every $x \in \Omega$,  $h_{\epsilon} * 1_{\Omega}{(x)} > c - \delta$.
Since $\delta>0$ was arbitrary, we have $h_\epsilon * 1_\Omega(x) \geq c$. Moreover, since $d_\epsilon$ is continuous as a function on $\mathbb R^n$, the bound holds for every $x \in \bar\Omega$.
\end{proof}

\subsection{Smoothing bounded measurable functions}\label{sec:boundedf}

In this subsection, we discuss convergence of $K_\epsilon f$ where $f$ is a bounded measurable function, and we use it to determine the limiting behavior of $K_\epsilon$ as $\epsilon \rightarrow 0$. The following proposition shows that for bounded measurable functions, the Lebesgue differentiation results discussed in \cref{sec:integral_operators} apply to more general families of kernels than the tophat kernels.

\begin{proposition}\label{prop:kernel_eq_lbsg}
     Let $h$ be a shape function satisfying \cref{prop:K1,prop:K2,prop:K4} and $K_{\epsilon}$ be the associated integral operator from \eqref{eq:k_op}. Then, for every bounded, measurable function $f : \mathbb R^n \to \mathbb R$,
$$
\lim _{\epsilon \rightarrow 0} K_{\epsilon}f(x) = \tilde f(x) \equiv \lim _{\epsilon \rightarrow 0} \frac{1}{|B_{\epsilon}(x)|} \int_{B_{\epsilon}(x)} f(y) \,d y
$$
for every $x \in \mathbb{R}^n$, provided that the limit on the right hand side exists.
\end{proposition}

\begin{proof}

Without loss of generality, we prove the statement for $x=0 \in \mathbb{R}^n$ and a positive bounded function $f$, i.e., we prove $ K_\epsilon f (0) =  h_{\epsilon}*f(0) \to \tilde f(0) $.
 Let $h^{(j)}$ be an increasing sequence of simple functions of the form $h^{(j)}(x)=\sum_{i=1}^{s_j} C_{j, i} 1_{B_{j, i}}$, $s_j \in \mathbb N$, that converge to $h$ from below such that $\lVert  h - h^{(j)}\rVert_{L^1(\mathbb R^n)} = \delta_j$.   Here, $B_{j, i} \subset \mathbb R^n$ is a ball centered at the origin, and
\begin{equation}
    \label{eq:simple_func}
    \sum_i^{s_j} C_{j, i}\left|B_{j, i}\right|=1-\delta_j
\end{equation}
with $\delta_j \to 0$. Note that since $\int_{\mathbb R^n} h(x) f(x)\, d x<\infty$ and $h^{(j)} f \nearrow hf$, by the monotone convergence theorem, $\int_{\mathbb{R}^n} h(x) f(x) \, d x=\lim _{j \to \infty} \int_{\mathbb{R}^n} h^{(j)}(x) f(x)\, d x$. This allows us to express the limit of $K_\epsilon f(0)$ as $\epsilon \to 0$ as
\begin{equation}
\begin{aligned}
    \nonumber \lim_{\epsilon\to 0} K_\epsilon f(0) &= \lim _{\epsilon \to 0} \int_{\mathbb{R}^n} h_{\epsilon}(x) f(x)\, d x =\lim _{\epsilon \to 0} \int_{\mathbb{R}^n} \epsilon^{-n} h(x / \epsilon) f(x)\, d x \\
                                         \label{eq:k_lim} & =\lim _{\epsilon \to 0} \lim _{n \to \infty} \int_{\mathbb{R}^n} h^{(j)}(y) f(y \epsilon)\, d y,
\end{aligned}
\end{equation}
where we made the change of variables $y=x/\epsilon$. Next, let $\epsilon=1/m$, $a_{j, m}=\int_{\mathbb{R}^n} h^{(j)}(y) f(y / m)\, d y$, and $a_m = \int_{\mathbb R^n} h(y) f(y/m) \, dy$.
To apply the Moore-Osgood theorem and interchange the order of limits in $j$ and $m$, we show the uniform convergence $a_{j,m} \to a_m$ as $j\to\infty$ with respect to $m \in \mathbb N$.  Let $M$ be an upper bound for $\lvert f\rvert$. For any $\tilde \epsilon > 0$ and $m \in \mathbb N$, there exists $j_* \in \mathbb N$ such that for all $j \geq j_*$, $\lVert h^{(j)}-h\rVert_{L^1\left(\mathbb{R}^n\right)} \leq \tilde\epsilon/M$. This implies,
\begin{displaymath}
    \left|a_{j, m}-a_m\right| =\left| \int_{\mathbb{R}^n}\left(h^{(j)}(y)-h(y)\right) f(y / m)\, d y \right| \leq M \int_{\mathbb{R}^n}\left|h^{(j)}(y)-h(y)\right|\, d y \leq \tilde\epsilon,
\end{displaymath}
and we conclude that the convergence of $a_{j,m}$ to $a_m$ is uniform with respect to $m \in \mathbb N$. After switching the order of limits in~\eqref{eq:k_lim}, we get
$$
\begin{aligned}
    \lim _{\epsilon \to 0} K_\epsilon f(0) &= \lim _{j \to \infty} \lim _{\epsilon \to 0} \int_{\mathbb{R}^n} h^{(j)}(y) f(y \epsilon)\, d y
                                                    =\lim _{j \to \infty} \lim _{\epsilon \to 0} \int_{\mathbb R^n} \epsilon^{-n}  h^{(j)}(x / \epsilon) f(x) \, dx \\
                                                   & =\lim _{j \to \infty} \lim _{\epsilon \to 0} \int_{\mathbb R^n} \epsilon^{-n} \sum_{i=1}^{s_j} C_{j, i} 1_{\epsilon B_{j, i}}(x)\, f(x)\, dx \\
                                                   & =\lim _{j \to \infty} \lim _{\epsilon \to 0}  \sum_{i=1}^{s_j} \epsilon^{-n} C_{j, i} \int_{\epsilon B_{j, i}} f(x)\, d x \\
                                                   & =\lim _{j \to \infty}  \sum_{i=1}^{s_j} C_{j, i} \left| B_{j, i} \right| \lim_{\epsilon \to 0} \frac{1}{\left| \epsilon B_{j, i} \right|} \int_{\epsilon B_{j, i}} f(x)\, d x\\
                                                   & =\lim _{j \to \infty}  \sum_{i=1}^{s_j} C_{j, i}\left|B_{j, i}\right| \tilde{f}(0)
=\tilde{f}(0),
\end{aligned}
$$
where we used~\eqref{eq:simple_func} to arrive at the last equality.
\end{proof}

\begin{corollary}\label{cor:degree_fn} Under the assumptions of \cref{prop:kernel_eq_lbsg}, the degree functions $d_\epsilon$ from~\eqref{eq:degree_fn} associated with $h$ and a measurable set $\Omega \subseteq \mathbb R^n$ satisfy $\lim_{\epsilon\to0} d_\epsilon(x) = \rho_\Omega(x)$ for  $x \in \Omega$, where $\rho_\Omega$ is the Lebesgue density function of $\Omega$ from~\eqref{eq:lebesgue_density}.
    \label{cor:degree_fn_lbsg}
\end{corollary}
\begin{proof}
    The claim follows by taking $f=1_{\Omega}$ and applying \cref{prop:kernel_eq_lbsg}.
\end{proof}

\Cref{prop:kernel_eq_lbsg} and \cref{cor:degree_fn_lbsg} lead to a pointwise convergence result for $P_{\epsilon}f$. In \cref{sec:proof_main}, we will prove a stronger version of the following corollary:

\begin{corollary}
    Let $\Omega$ be a thick open set in $\mathbb{R}^n$. Let $h \in L^1(\mathbb R^n)$ be a function satisfying \cref{prop:K1,prop:K2,prop:K4} and $P_{\epsilon}$ the associated Markovian integral operator from~\eqref{eq:p_op}. Then, for every bounded measurable function $f: \Omega \to \mathbb R$ and point $x$ in the interior of $\Omega$ we have, $\lim_{\epsilon \rightarrow 0}P_{\epsilon}f(x) = \tilde{f}(x)$, with $\tilde f$ defined as in \cref{prop:kernel_eq_lbsg}.
     In particular, $\lim_{\epsilon \rightarrow 0}P_{\epsilon}f(x) = f(x)$ a.e.\ on $\bar\Omega$.
\end{corollary}

\begin{proof}
    Recall that $P_\epsilon f(x) = K_\epsilon f(x) / d_\epsilon(x)$, where $d_\epsilon(x)$ is bounded away from zero by \cref{lem:deps_lower_bound}. Moreover, by \cref{cor:degree_fn_lbsg} we have $\lim_{\epsilon \to 0}d_\epsilon(x) = \rho_\Omega(x) = 1$ for every $x$ in the interior of $\Omega$. The claim then follows by extending $f$ by zero to a function on $\mathbb{R}^n$ and applying \cref{prop:kernel_eq_lbsg} to obtain $ \lim_{\epsilon \rightarrow 0} K_\epsilon f(x) = \tilde{f}(x) / \rho_\Omega(x) = \tilde f(x)$.
\end{proof}

\subsection{Proof of \cref{thm:main}}
\label{sec:proof_main}

We seek analogs of \cref{thm:A,thm:B} for the case where convolution by the function $h_\epsilon$ (equivalently, integration against the translation-invariant kernel $k_\epsilon$) is replaced by integration against the Markov-normalized kernel $p_\epsilon$ from~\eqref{eq:p_kernel} over the thick open set $\Omega \subseteq\mathbb{R}^n$. Note that for a generic $\Omega$, the kernel $p_{\epsilon}$ is no longer translation-invariant and it has non-trivial behavior near the boundary of $\Omega$. In particular, $P_{\epsilon}f(x)$ can blow up if $d_{\epsilon}(x) \rightarrow 0$ at $x\in \partial \Omega$.

Let $\Omega$ be a thick open set and $h \in L^1(\mathbb R^n)$ a shape function. For $\epsilon_0>0$ chosen according to \cref{def:thick}, $p_\epsilon : \Omega \times \Omega \to \mathbb R$ defined in~\eqref{eq:p_kernel}, and $f \in L^p(\Omega)$, we define the Markovian local maximal function on $\Omega$ by
\begin{displaymath}
    {M}^{\Omega}_h f(x)=\sup _{0<\epsilon<\epsilon_0}\left|\int_{\Omega} p_{\epsilon}(x, y) f(y)\,  d y\right|.
\end{displaymath}
Note that, unlike the classical definition in~\eqref{eq:maximal}, the range for the parameter $\epsilon$ is restricted to the interval $\left(0,\epsilon_0\right)$ where we have a strictly positive lower bound for $l_\epsilon * 1_\Omega$. We then have:

\begin{proposition}
    \label{thm:A2}
    Let $\Omega \subseteq \mathbb R^n$ be a thick open set and $h \in L^1(\mathbb R^n)$ a shape function satisfying \cref{prop:K1,prop:K2,prop:K3,prop:K4}. Then, for every $p \in (1,\infty]$ there exists $\tilde A>0$ such that for every $f\in L^p\left(\Omega\right)$ we have
$$
\left\|M^{\Omega}_h f\right\|_{L^p(\Omega)} \leq \tilde A\|f\|_{L^p(\Omega)}.
$$
\end{proposition}

\begin{proof}
By \cref{lem:deps_lower_bound}, thickness of $\Omega$ implies that there exists $c>0$ such that $h_\epsilon*1_{\Omega}(x) \geq c>0$ for every $\epsilon \in (0,\epsilon_0)$ and $x \in \Omega$.
Using the definition of the maximal function $M_h f$ in~\eqref{eq:maximal}, we get
$$
\left|\int_{\Omega} p_{\epsilon}(x, y) f(y)\,  d y\right|= \frac{\left\lvert h_{\epsilon} * f(x) \right\rvert}{h_\epsilon*1_{\Omega}(x)} \leq c^{-1}\lvert h_{\epsilon} * f(x)\rvert \leq c^{-1} M_h f(x),
$$
for every $x \in \Omega$. Meanwhile, using \cref{thm:A}, we have $\lVert M_h f \rVert_{L^p(\Omega)} \leq A \lVert f\rVert_{L^p(\Omega)}$, where we have identified $f$ with a function on $\mathbb R^n$ that vanishes on $\mathbb R^n \setminus \Omega$. Combining the last two bounds, we obtain
\begin{displaymath}
    \lVert {M}^{\Omega}_h f \rVert_{L^p(\Omega)} \leq c^{-1} \lVert M_h f\rVert_{L^p(\Omega)} \leq c^{-1} A \lVert f\rVert_{L^p(\Omega)},
\end{displaymath}
and the claim of the proposition follows with $ \tilde A = c^{-1} A$.
\end{proof}

\begin{corollary}
    \label{cor:p_op_lp}
    The integral operators $\{ P_\epsilon \}_{0<\epsilon<\epsilon_0}$ from~\eqref{eq:p_op} are well-defined and uniformly bounded on $L^p(\Omega)$ for every $p \in [1,\infty]$.
\end{corollary}

\begin{proof}
    The well-definition and uniform boundedness of $P_\epsilon : L^1(\Omega) \to L^1(\Omega)$ for $0< \epsilon<\epsilon_0$ follows by construction of the kernel $p_\epsilon$ in~\eqref{eq:p_kernel} under \cref{prop:K1,prop:K2,prop:K3,prop:K4} and thickness of $\Omega$; see \cref{sec:integral_operators,sec:markov_normalization}. In the case $p\in (1,\infty]$, for any $\epsilon \in (0,\epsilon_0)$ and $f \in L^p(\Omega)$ we have $\lVert P_\epsilon f\rVert_{L^p(\Omega)} \leq \lVert M_h^\Omega f \rVert_{L^p(\Omega)}$. The claim follows from \cref{thm:A2} with $P_\epsilon: L^p(\Omega) \to L^p(\Omega)$ bounded in operator norm by $\tilde A$.
\end{proof}

\begin{proposition}
    \label{thm:B2}
    Let $\Omega \subseteq \mathbb R^n$ be a thick open set and $h $ a shape function satisfying \cref{prop:K1,prop:K2,prop:K3,prop:K4} and the decay bound $\lvert h(x)\rvert \leq C(1+ \lVert x\rVert)^{-(n+\delta)}$ for some $C,\delta > 0$. Let $P_\epsilon$ be the induced Markovian integral operators from~\eqref{eq:p_op}. Then, for every $f\in L^p\left(\Omega\right)$ with $1<p\leq \infty$, $P_\epsilon f$ converges to $f$ a.e.\ as $\epsilon\to 0^+$.
\end{proposition}

\begin{proof}
    By \cref{cor:degree_fn}, we have $\lim_{\epsilon\to 0^+} h_\epsilon*1_{\Omega}(x) = 1$ for every interior point $x \in \Omega$.  and by \cref{thm:B} $\lim_{\epsilon\to 0^+} h_\epsilon * f = f(x)$ for a.e.\ $x \in \Omega$. Thus, for a.e.\ $x \in \Omega$, we have
    \begin{displaymath}
        \lim_{\epsilon\to 0^+} P_\epsilon f(x) = \frac{\lim_{\epsilon \to 0^+} h_\epsilon * f(x)}{\lim_{\epsilon\to 0^+}h_\epsilon*1_{\Omega}(x)} = f(x).
    \end{displaymath}
\end{proof}

\begin{remark}
    The thickness of $\Omega$ is an important constraint for $h_{\epsilon}*f $ to lie in $ L^p(\Omega)$ and here is another way of looking at this condition. By \cref{lem:deps_lower_bound}, thickness of $\Omega$ is equivalent to the uniform boundedness of
    $$ \frac{1}{l_{\epsilon}*1_{\Omega}(x)}  = \frac{|B_{\epsilon}(x)|}{|B_{\epsilon}(x)\cap\Omega|} \in L^{\infty}(\Omega),$$
    and we can then conclude that
    $$
    \left\lVert \frac{1}{d_{\epsilon}}\cdot h_{\epsilon}*f \right\rVert_{L^{p}(\Omega)} \leq c^{-1}  \| h_{\epsilon}*f \|_{L^{p}(\Omega)} \leq c^{-1}\, A\,  \| f \|_{L^{p}(\Omega)}.
    $$
    However, multiplication by $\frac{1}{d_{\epsilon}}$ is a bounded operator on $L^p(\Omega)$ if and only if $\frac{1}{d_{\epsilon}}$ lies in $L^{\infty}(\Omega)$. This suggests that if $\Omega$ is not thick then $P_{\epsilon}$ may not even be a bounded operator on $L^{p}(\Omega)$.
\end{remark}

We are now ready to prove \cref{thm:main}.

\begin{proof}[Proof of \cref{thm:main}]
    The well-definition and boundedness of $P_\epsilon : L^p(\Omega) \to L^p(\Omega)$ for $p\in[1,\infty]$ is proved by \cref{cor:p_op_lp}. Claim~1 is proved by \cref{thm:B2}.

    Turning to Claim~2, by the dominated convergence theorem for $L^p(\Omega)$, if the sequence $f_j := P_{1/j} f \in L^p(\Omega)$ converges as $j \to \infty$ a.e.\ to $f$ and the terms $f_j$ are dominated in by some function in $L^p(\Omega)$, then $f_j$ converges to $f$ in $L^p(\Omega)$ norm. Thus, since we have a.e.\ convergence of $P_\epsilon f$ to $f$ from Claim~1, to prove Claim~2 it suffices to show that $f_j \leq F$ for some $F \in L^p(\Omega)$. Indeed, by \cref{thm:A2}, we have that for every $\epsilon \in (0, \epsilon_0)$ $g_\epsilon$ is bounded in $L^p(\Omega)$ by $M^{\Omega}_h f$. Setting $F = M^\Omega_h$ thus proves Claim~2.

    Next, to prove Claim~3 we use \cref{lem:conv_contin} to deduce that $K_\epsilon f$ lies $C_b(\Omega)$ for every $f \in L^p(\Omega)$ with $\frac{1}{p} + \frac{1}{q} = 1$. That $P_\epsilon f = \frac{K_\epsilon f}{d_\epsilon}$ also lies in $C_b(\Omega)$ then follows from the fact that $\frac{1}{d_\epsilon} \in C_b(\Omega)$ by \cref{lem:deps_lower_bound}. Finally, to prove Claim~4, if $h \in L^1(\mathbb R^n) \cap H^r(\mathbb R^n)$ it follows from \cref{lem:sobolev} that for every $f \in L^2(\Omega)$ $K_\epsilon f$ lies in $C^r(\Omega)$. If $\Omega$ is bounded, we have $d_\epsilon \in C^r(\Omega)$ since $1_\Omega \in L^2(\Omega)$, and since, by \cref{lem:deps_lower_bound}, $d_\epsilon(x) \geq c >0$ on $\Omega$ we also have $\frac{1}{d_\epsilon} \in C^r(\Omega)$. Thus, $P_\epsilon f = \frac{K_\epsilon f}{d_\epsilon}$ lies in $C^r(\Omega)$.
\end{proof}

\section{Thick sets as metric measure spaces}
\label{sec:metric_measure_spaces}

Recall that a metric measure space $(X,d,\mu)$ is doubling \cite{Heinonen01} if there exists $M>0$ such that for any $r>0$ and $x\in X$,
\begin{equation}
    \label{eq:doubling}
    \mu(B_{2r}(x)) \leq M \mu(B_{r}(x)).
\end{equation}
The base-2 logarithm of the optimal such constant $M$ is known as the doubling dimension of $(X, d, \mu)$. We will say that $(M, d, \mu)$ is locally doubling if \eqref{eq:doubling} holds for all $r \in (0, r_0)$ for some $r_0>0$. Given a measurable set $\Omega \subseteq \mathbb{R}^n$, we consider the measure space $(\Omega, d, \mu)$, where $d$ and $\mu$ are the Lebesgue measure Euclidean distance induced from $\mathbb{R}^n$, respectively.

\begin{proposition}
    If $\Omega \subseteq \mathbb R^n$ is a thick measurable set then $(\Omega,d,\mu)$ is locally doubling.
    \label{prop:doubling}
\end{proposition}

\begin{proof}
Since $\Omega$ is thick, there exist $r_0>0$ and $c>0$ such that for all $r\in (0, r_0)$ $|B_{r}(x) \cap \Omega| \geq c |B_{r}(x)|$. Therefore,
$$
\mu(B_{2r}(x)) = |B_{2r}(x) \cap \Omega| \leq |B_{2r}(x)| \leq M |B_{r}(x)| \leq \frac{M}{c} |B_{r}(x) \cap \Omega| = \frac{M}{c} \mu(B_{r}(x)),
$$
showing that $\Omega$ satisfies the doubling inequality for $r<r_0$.
\end{proof}

The following proposition establishes that thickness and the doubling property are independent notions.
\begin{proposition}
    \label{prop:not_doubling}\

 \begin{enumerate}
     \item There exist thick measurable sets $\Omega \subset \mathbb R^n$ such that $(\Omega,d,\mu)$ is not doubling.
     \item There exist doubling metric measure spaces $(\Omega, d, \mu)$, $\Omega \subset \mathbb R^n$, such that $\Omega$ is not thick.
 \end{enumerate}
 \end{proposition}

 \begin{proof}
     Starting from Claim~1, we build explicit examples of thick measurable sets that are not doubling by considering the following family of subsets of $\mathbb R^n$. Let $x \in \mathbb{R}^n$ and for $0\leq \alpha, \beta < \frac{1}{2}$ define
     $$A(x; r,\alpha, \beta) = B_r(x)\setminus \left( B_{r-\alpha r}(x)\setminus B_{\beta r}(x) \right).$$
     In two dimensions, these sets can be constructed by removing an annulus from a disc of radius $r$ centered at $x$. For simplicity, for the rest of this proof we set $n=1$.

     Choosing a countable collection of $x_j, r_j \in \mathbb R$ such that the corresponding balls, $B_{r_{j}}(x_j)$, are disjoint, we define $\Omega \subset \mathbb R$ as the disjoint union $\Omega = \bigcup_{j=1}^{\infty} A(x_j; r_j,\alpha_j, \beta_j)$, with $0\leq \alpha_j, \beta_j < \frac{1}{2}$ and $r_j>0$. Making the observation
     $$
      \frac{\mu(B_{r_j} (x_j))}{\mu(B_{r_j/2} (x_j))} = \frac{r_j - (r_j-\alpha_j r_j) + \beta_j r_j}{\beta_j r_j} = 1 + \frac{\alpha_j}{\beta_j},
      $$
     it is not hard to see that $(\Omega, d, \mu)$ is doubling only if $\alpha_j / \beta_j$ is a bounded sequence. On the other hand, setting $\alpha_j=1/3$, $\beta_j=1/(3j)$, and $r_j =j$, then $\Omega$ is thick (since each connected component of $\Omega$ has length at least $1/3$) but does not satisfy the doubling property at the center points $x_j$ as $j\rightarrow \infty$ (see the figure bellow).\\

     \begin{tikzpicture}[scale=1.5]
  % Draw x-axis
  \draw[thin, ->] (-1,0) -- (6.5,0) node[right] {$\dots$};

  % Draw line segments
  \draw[ultra thick, blue] (-1,0) -- (-0.66,0);
  \draw[ultra thick, blue] (-0.33,0) -- (0.33,0);
  \draw[ultra thick, blue] (0.66,0) -- (1.0,0);
  \draw[ultra thick, blue] (2,0) -- (2.66,0);
  \draw[ultra thick, blue] (3.66,0) -- (4.33,0);
  \draw[ultra thick, blue] (5.33,0) -- (6.0,0);

  % Draw points
  \foreach \x in {-1,-2/3,-1/3,1/3,2/3,1, 2, 2.66, 3.66, 4.33, 5.33, 6} {
    \fill[blue] (\x,0) circle (1.5pt);
  }

 \foreach \x in {0,4} {
    \fill[blue] (\x,0) circle (1.5pt);
  }

  % Label points
  \foreach \x in {0,4} {
    \node[below] at (\x,-0.1) {$\x$};
  }
    \draw[decorate,decoration={brace,amplitude=10pt},xshift=0pt,yshift=5pt] (-1,0) -- (1,0) node[midway,yshift=15pt] {$A(0;1,\frac{1}{3},\frac{1}{3})$};
    \draw[decorate,decoration={brace,amplitude=10pt},xshift=0pt,yshift=5pt] (2,0) -- (6,0) node[midway,yshift=15pt] {$A(4;2,\frac{2}{3},\frac{1}{3})$};
\end{tikzpicture}
     %$\Omega$ contains intervals of the form $(x_j-\frac{j}{3}, x_j+\frac{j}{3})$

     %thickness of $\Omega$ is equivalent to \dgm{Not obvious without additional information on $\alpha_j, \beta_j, r_j$.} $\inf_{j \in \mathbb N} \{ \alpha_j r_j, \beta_j r_j \} > 0$. Setting $\alpha_j=1/3$, $\beta_j=1/(3j)$, and $r_j =j$, it follows that $\Omega$ is thick but does not satisfy the doubling property.

     % If we fix $x$ and consider $ A(x; r,\alpha, \beta) \subset \mathbb{R}^n$ as a metric measure space then \dgm{How does this statement relate to the rest of the proof?}
     % $$
     % \frac{\mu(B_r (x))}{\mu(B_{r/2} (x))} = \frac{r - (r-\alpha r) + \beta r}{\beta r} = 1 + \frac{\alpha}{\beta}
     % $$

    % \dgm{Does this hold for any $n$? Is the above example supposed to demonstrate the equivalence?}
    %  \dgm{This is too vague. How are the $x_j$ and $r_j$ chosen? Is this for dimension $n=1$ or any dimension?} Using the fact that

     To prove Claim~2, let $\Omega=\{ x=(x_1,x_2)\in \mathbb{R}^2:\, 0\leq x_1, \, 0\leq x_2 \leq \frac{1}{x_1}  \}$. This set is not thick, however, $(\Omega,d,\mu)$ is doubling with doubling dimension 2.
 \end{proof}

\section{Bistochastic kernels}
\label{sec:bistoch}

In this subsection we consider an alternative kernel normalization procedure than the approach described in \cref{sec:markov_normalization}, leading to a bistochastic kernel whose corresponding integral operators preserve constant functions as well as integrals of functions. Our approach is based on the class of bistochastic kernels studied by Coifman and Hirn in \cite{CoifmanHirn13} which was adapted in \cite{DasEtAl21} to a setting closer to the function smoothing problem studied in this paper.

Assuming, as in \cref{sec:markov_normalization}, that $\Omega \subseteq \mathbb R^n$ is a thick open set so \cref{prop:d_func} holds for every $\epsilon\in (0, \epsilon_0)$, $\epsilon_0>0$, we define the normalization function $q_\epsilon: \mathbb R^n \to \mathbb R_+$ as
\begin{equation}
    \label{eq:q_func}
    q_{\epsilon} = K_\epsilon \left(\frac{1_\Omega}{d_\epsilon}\right) \equiv \int_{\Omega} \frac{k_{\epsilon}(x,y)}{d_{\epsilon}(y)} \, dy.
\end{equation}
Since $\frac{1_\Omega}{d_\epsilon} \in L^\infty(\mathbb R^n)$, it follows from \cref{lem:conv_contin} that $q_\epsilon \in C_b(\mathbb R^n)$. In particular, since $c \leq d_\epsilon \leq 1$ on $\Omega$ (see \cref{lem:deps_lower_bound}), we have
$$
c \leq K_\epsilon 1_{\Omega}(x) \leq K_\epsilon \left(\frac{1_{\Omega}}{d_{\epsilon}(x)}\right) \leq \frac{1}{c} K_\epsilon 1_\Omega(x) \leq \frac{1}{c},
$$
and we deduce that $q_\epsilon$ is bounded above and bounded away from 0 from below, $c \leq q_\epsilon \leq c^{-1}$. As a result, $\tilde p_\epsilon: \Omega \times \Omega \to \mathbb R$ with
\begin{equation}
    \label{eq:p_bistoch}
    \tilde{p}_{\epsilon}(x,y) = \int_{\Omega} \frac{k_{\epsilon}(x,z) k_{\epsilon}(z,y)}{d_{\epsilon}(x) q_{\epsilon}(z) d_{\epsilon}(y)} \, dz
\end{equation}
is a well-defined, measurable kernel. Note the use of right normalization of $k_\epsilon$ in the definition~\eqref{eq:q_func} of $q_\epsilon$ as opposed to left normalization employed in~\eqref{eq:p_kernel} to build the kernel $p_\epsilon$.

One readily verifies that
\begin{displaymath}
    \tilde p_\epsilon(x,y) = \tilde p_\epsilon(y, x), \quad \tilde p_\epsilon(x, y) \geq 0, \quad \int_{\mathbb R^n} \tilde p_\epsilon(x, z) \, dz = 1,
\end{displaymath}
for every $x, y \in \Omega$; that is, $\tilde p_\epsilon$ is a symmetric (bistochastic) Markov kernel on $\Omega$. Moreover, defining $\tilde P_\epsilon : L^1(\mathbb R^n) \to L^1(\mathbb R^n)$ as
\begin{equation}
    \label{eq:p_bistoch_op}
    \tilde{P}_{\epsilon}f(x) = \int_{\Omega} \tilde{p}_{\epsilon}(x,y) f(y)\, dy,
\end{equation}
if follows from symmetry of $\tilde p_\epsilon$ that $\tilde P_\epsilon$ preserves integrals of functions,
\begin{equation}
    \label{eq:mass_preserve}
    \int_{\Omega} \tilde P_\epsilon f(x) \, dx = \int_\Omega f(x) \, dx, \quad \forall f \in L^1(\Omega).
\end{equation}
In addition, we have $\lVert \tilde P_\epsilon f\rVert_{L^1(\Omega)} \leq \lVert f\rVert_{L^1(\Omega)}$ so $\tilde P_\epsilon$ is a contraction on $L^1(\Omega)$. If $f \in L^1(\mu)$ is the density of a measure $\nu_f$ on $\Omega$ with respect to Lebesgue measure, \eqref{eq:mass_preserve} implies that $\nu_f(\Omega) = \nu_{\tilde P_\epsilon f}(\Omega)$ so we can think of $\tilde P_\epsilon$ as a mass-preserving operator.

The following is an analogous convergence result to \cref{thm:main} for integral operators induced by the bistochastic kernels from~\eqref{eq:p_bistoch}.

\begin{theorem}
    \label{thm:bistoch}
    Let $\Omega$ be a thick open subset of $\mathbb{R}^n$, and suppose that the shape function $h$ satisfies \cref{prop:K1,prop:K2,prop:K4,prop:K3}. Then, there exists $\epsilon_0 >0$ such that for every $p \in [1, \infty]$, $\{\tilde P_\epsilon\}_{0<\epsilon<\epsilon_0}$ from~\eqref{eq:p_bistoch_op} is a uniformly bounded family of operators from $L^p(\Omega)$ to itself. In addition, if $\Omega$ is bounded and $h$ satisfies the decay bound $\lvert h(x) \rvert \leq C (1+ \lVert x\rVert)^{-(n+\delta)}$ for some $C,\delta > 0 $, the following hold.
    \begin{enumerate}
        \item For every  $p \in (1, \infty)$ and $f \in L^p(\Omega)$, $\tilde P_\epsilon f$ converges to $f$ in $L^p(\Omega)$ norm.
        \item If $h$ lies in $L^1(\mathbb R^n) \cap L^q(\mathbb R^n)$ with $\frac{1}{p}+\frac{1}{q} = 1$, then for every $p \in (1,\infty]$ and $f \in L^p(\Omega)$, $\tilde P_\epsilon f$ lies in $C_b(\Omega)$.
        \item If $h \in L^1(\mathbb R^n) \cap H^r(\mathbb R^n)$, then for every $f \in L^2(\Omega)$ $P_\epsilon f$ lies in $C^r(\Omega)$.
    \end{enumerate}
\end{theorem}

Note that \cref{thm:bistoch} is somewhat weaker than \cref{thm:main} in that it requires that $\Omega$ is bounded and does not make a claim about pointwise convergence of $\tilde P_\epsilon f$ to $f$. For the proof of \cref{thm:bistoch} we will make use of the following result.

\begin{lemma}\label{lem:mult_conv_s}
    Let $\Omega \subseteq \mathbb R^n$ be a measurable set, and consider $f \in L^\infty(\Omega)$ and $ \{ f_\epsilon \in L^\infty(\Omega) \}_{0<\epsilon<\epsilon_0}$ where the $f_\epsilon$ are uniformly bounded in $L^\infty(\Omega)$. For $p \in [1, \infty)$, let $M_f : L^p(\Omega) \to L^p(\Omega)$ (resp.\ $M_{f_\epsilon} : L^p(\Omega) \to L^p(\Omega)$) be the multiplication operator by $f$ (resp.\ $f_\epsilon$). Then, if $f_\epsilon$ converges in $L^p(\Omega)$ to $f$ as $\epsilon \to 0$, $M_{f_\epsilon}$ converges to $M_f$ in the strong operator topology of $B(L^p(\Omega))$.
\end{lemma}

\begin{proof}
    See \cref{app:proofs}.
\end{proof}

\begin{proof}[Proof of \cref{thm:bistoch}]
    Observe that
    \begin{equation}
        \tilde P_\epsilon = P_\epsilon \circ M_{1/q_\epsilon} \circ P_\epsilon \circ M_{1/d_\epsilon},
        \label{eq:p_combo}
    \end{equation}
    where $M_{1/d_\epsilon}$ and $M_{1/q_\epsilon}$ are the multiplication operators by $1/d_\epsilon$ and $1/q_\epsilon$, respectively. Since $d_\epsilon \leq c$ and $q_\epsilon \geq c$, we have that for every $p \in [1, \infty]$, $ \{ M_{1/d_\epsilon} \}_{0<\epsilon<\epsilon_0}$ and $ \{ M_{1/q_\epsilon} \}_{0<\epsilon<\epsilon_0}$ are uniformly bounded families of operators on $L^p(\Omega)$. Since, by \cref{thm:main}, $ \{ P_\epsilon \}_{0<\epsilon<\epsilon_0}$ is also uniformly bounded, it follows that $ \{ \tilde P_\epsilon \}_{0<\epsilon<\epsilon_0}$ is uniformly bounded on $L^p(\Omega)$, as claimed.

    Consider now Claim~1. Since $\Omega$ is bounded and $c \leq d_\epsilon \leq 1$, $d_\epsilon$ and $1/d_\epsilon$ lie in $L^p(\Omega)$ for all $p \in [1,\infty]$. By \cref{cor:lp_conv_k}, for every $p \in (1,\infty)$, $d_\epsilon = K_\epsilon 1_\Omega$ converges to $1_\Omega$ in $L^p(\Omega)$. Moreover, we have
    \begin{displaymath}
        \left\lVert 1_\Omega - \frac{1}{d_\epsilon} \right\rVert_{L^p(\Omega)} = \left\lVert\frac{d_\epsilon - 1_\Omega}{d_\epsilon} \right\rVert_{L^p(\Omega)} = \left\lVert\frac{K_\epsilon 1_\Omega - 1_\Omega}{d_\epsilon} \right\rVert_{L^p(\Omega)} \leq c^{-1} \lVert K_\epsilon 1_\Omega - 1_\Omega\rVert_{L^p(\Omega)},
    \end{displaymath}
    and it follows that $1/d_\epsilon$ also converges to $1_\Omega$ in $L^p(\Omega)$. By \cref{lem:mult_conv_s}, we deduce that, as $\epsilon\to 0$, $M_{1/d_\epsilon}$ converges strongly to the identity on $L^p(\Omega)$. One can similarly deduce from the facts that $c \leq q_\epsilon \leq c^{-1}$ and $q_\epsilon = K_\epsilon( \frac{1}{d_\epsilon} ) $ that $M_{1/q_\epsilon}$ converges strongly to the identity on $L^p(\Omega)$. Since the operators $P_\epsilon$, $M_{1/q_\epsilon}$, and $M_{1/d_\epsilon}$ are uniformly bounded in $B(L^p(\Omega))$ for $\epsilon \in (0,\epsilon_0) $, it follows from~\eqref{eq:p_combo} that $\tilde P_\epsilon$ converges strongly to the identity; i.e., $\tilde P_\epsilon f$ converges to $f$ in $L^p(\Omega)$, as claimed.

    Turning to Claim~2, since $M_{1/q_\epsilon}$, $P_\epsilon$, and $M_{1/d_\epsilon}$ are all bounded operators, we have that for every $ f \in L^p(\Omega)$, $g_\epsilon = M_{1/q_\epsilon} P_\epsilon M_{1/d_\epsilon} f$ lies in $L^p(\Omega)$, and thus $\tilde P_\epsilon f = P_\epsilon g_\epsilon$ lies in $C_b(\Omega)$ by \cref{thm:main}. This proves the claim.

    Finally, Claim~3 follows again from \cref{thm:main} in conjunction with the fact that $g_\epsilon$ from above lies in $L^2(\Omega)$ for every $f \in L^2(\Omega)$.
\end{proof}

\section{Numerical approximation}
\label{sec:numerical_approximation}

In this section, we consider the problem of numerically approximating the action of the integral operators introduced in \cref{sec:mathematical_framework,sec:bistoch} on a class of piecewise-continuous functions on a thick, bounded, open domain $\Omega \subset \mathbb R^2$. Specifically, we consider a finite collection of polygons $S_\ell \subseteq \Omega$ indexed by $\ell \in \Lambda$ and functions $f: \Omega \to \mathbb R$ of the form
\begin{equation}
    f(x) = \sum_{\ell\in \Lambda} f_\ell(x),
    \label{eq:f_orig}
\end{equation}
where $f_\ell : \Omega \to \mathbb R$ are functions with $\supp f_\ell \subseteq S_\ell$ and $f_\ell \rvert_{S_\ell}$ bounded continuous. This class of functions is appropriate for the numerical experiments presented in \cref{sec:experiments} where the polygons $S_\ell$ correspond to ice floes as represented within two-dimensional sea ice DEMs and the functions $f_\ell$ represent physical floe variables such as mass density; see \cref{fig:floe} for a schematic.

\begin{figure}
    \centering
      \includegraphics[width=0.8\linewidth]{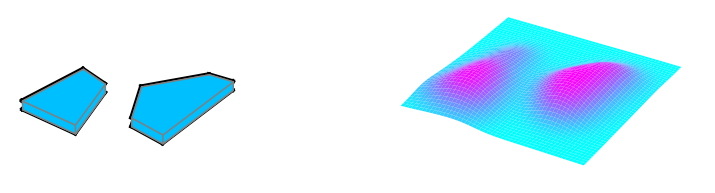}
     \caption{Piecewise-continuous mass density function $m(x)$ associated with two ice floes (left) and the corresponding smooth mass density field obtained by smoothing $m$ by a Markovian integral operator (right).}
      \label{fig:floe}
\end{figure}

Our goal is to approximate functions in this class by kernel smoothed-functions $\hat g_\epsilon : \Omega \to \mathbb R$ where $\hat g_\epsilon$ is a continuous function that numerically approximates $P_\epsilon f$ obtained from the Markov integral operator in \eqref{eq:p_op} (or its bistochastic variant, $\tilde P_\epsilon$ from~\eqref{eq:p_bistoch_op}). As a concrete example, for an input function $f$ from~\eqref{eq:f_orig} the corresponding kernel-smoothed functions $g_\epsilon := P_\epsilon f$ take the form
\begin{equation}
    g_\epsilon(x) = \frac{1}{d_\epsilon(x)} \sum_{\ell\in \Lambda} \int_{S_\ell} k_\epsilon(x,y) f_\ell(y)\, dy, \quad d_\epsilon(x) = \int_\Omega k_\epsilon(x,y) \, dy.
    \label{eq:g_integral2}
\end{equation}
By \cref{thm:main}, we have $g_\epsilon \in C_b(\Omega)$ since $f$ lies in $L^\infty(\Omega)$ by boundedness of $\Omega$ and piecewise-continuity of $f$.

Let $\eta_{\epsilon,x} : \Omega \to \mathbb R$ denote the section of the kernel $k_\epsilon$ at $x \in \Omega$, defined as $\eta_{\epsilon,x}(y) = k_\epsilon(x,y)$. From \eqref{eq:g_integral2}, we see that in order to approximate $g_\epsilon(x)$ at some $x \in \Omega$ it is sufficient to approximate integrals of (i) kernel-weighted functions $y \mapsto \eta_{\epsilon,x}(y) f_\ell(y)$ over the subdomains $S_\ell$; and (ii) the kernel section $y \mapsto \eta_{\epsilon,x}(y)$ over the domain $\Omega$. In \cref{sec:integration_scheme,sec:monte_carlo,sec:triangulation}, we describe methods for approximating these integrals using numerical quadrature. These methods can be modified to deal with higher-dimensional domains and/or other integral operators than $P_\epsilon$ (in particular, the bistochastic operators $\tilde P_\epsilon$).

\subsection{Quadrature scheme}
\label{sec:integration_scheme}

We supplement \cref{prop:K1,prop:K2,prop:K3,prop:K4} required of the shape function $h$ with two further assumptions:
\begin{enumerate}[label=(K\arabic*)]
    \setcounter{enumi}{4}
    \item \label[prop]{prop:K5} $h$ is continuous.
    \item \label[prop]{prop:K6} $h$ is strictly positive.
\end{enumerate}

With these assumptions, given an integration domain $Y \subseteq \Omega$ (which could be one of the subdomains $S_\ell$, or the entire set $\Omega$) and an integrand $\varphi_x: Y \to \mathbb R$ parameterized by $x \in \Omega$ (which could be one of the kernel-weighted functions $\eta_{\epsilon,x} f_\ell$ or the kernel section $\eta_{\epsilon,x}$), we approximate the integral $I [\varphi_x] := \int_Y \varphi_x(y) \, dy$ by
\begin{equation}
    I_N [\varphi_x] := \sum_{j=1}^N w_{j,N} \, \varphi_x(y_{j,N}).
    \label{eq:approx_integral}
\end{equation}
Here, $y_{1,N}, \ldots, y_{N,N}$ are quadrature nodes in $Y$ and $w_{1,N},\ldots, w_{N,N} \in \mathbb R_+$ are corresponding weights. Applying~\eqref{eq:approx_integral} to each integral in~\eqref{eq:g_integral2} leads to the approximation
\begin{equation}
    \hat g_\epsilon(x) \equiv g_{\epsilon,\bm N}(x) := \frac{1}{d_{\epsilon,N_\Omega}(x)} \sum_{\ell \in \Lambda_t} \tilde g_{\epsilon,\ell,N_\ell}(x),
    \label{eq:g_quadrature}
\end{equation}
where $\bm N = (N_\Omega,N_{\ell\in \Lambda})$ is a multi-index parameterizing numerical quadrature over the domains $\Omega$ and $\{ S_\ell \}_{\ell \in \Lambda_t}$, respectively, and $d_{\epsilon,N_\Omega}: \Omega \to \mathbb R$ and $\tilde g_{\epsilon,\ell,N_\ell} : \Omega \to \mathbb R$ are given by
\begin{displaymath}
    d_{\epsilon,N_\Omega}(x) = I_{N_\Omega}[\eta_{\epsilon,x}], \quad \tilde g_{\epsilon,\ell,N_\ell}(x) = I_{N_\ell}[\eta_{\epsilon,x} f_{t,\ell}].
\end{displaymath}
Note that by \cref{prop:K5,prop:K6}, $d_{\epsilon,N_\Omega}$, $\tilde g_{\epsilon,\ell,N_\ell}$, and thus $g_{\epsilon,\bm N}$, are well-defined, bounded continuous functions on $\Omega$. If, in addition, the shape function $h$ lies in $H^r(\Omega)$, then by \cref{thm:main} $g_{\epsilon,\bm N}$ lies in $C^r(\Omega)$. We seek to choose the quadrature nodes and weights used in~\eqref{eq:g_quadrature} such that $g_{\epsilon, \bm N}$ converges to $g_{\epsilon}$ in a limit of $\bm N \to \infty$, uniformly with respect to $x \in \Omega$.

There are several possible quadrature schemes that meet these objectives, of which we describe two examples: Monte Carlo methods and triangulation methods for polygonal domains.

\subsection{Monte Carlo methods}
\label{sec:monte_carlo}

In the Monte Carlo approach, we place each target integration domain $Y$ in a bounding rectangle $R = L_1 \times L_2  \supseteq Y$ (where $L_1$ and $L_2$ are intervals), draw $M$ independent random samples $\tilde y_1, \ldots, \tilde y_M \in R$ from the Lebesgue measure on $R$, and set the quadrature nodes $y_{1,N}, \ldots, y_{N,N}$ to any subset of $\{ \tilde y_1, \ldots, \tilde y_M \}$ that lies in $Y$ (we assume that $M$ is sufficiently large for the chosen $N$). The corresponding weights are chosen uniformly as $w_{j,N} = \lvert R\rvert/M$ where $\lvert R\rvert$ is the area of $R$. An application of the central limit theorem then leads to the classical Monte Carlo error estimate
\begin{equation}
    \left\lvert I_N[\varphi_x] - I[\varphi_x]\right\rvert \simeq \frac{s_x \lvert R\rvert}{\lvert Y\rvert} N^{-1/2}, \quad s_x^2 = \int_R ( \varphi_x(y) - (I [\varphi_x]))^2 \, dy,
    \label{eq:mc_estimate}
\end{equation}
which holds for large $N$ with high probability; e.g., \cite{Caflisch98}.

The constant $s_x$ in~\eqref{eq:mc_estimate} is proportional to the standard deviation of the integrand $\varphi_x$ on the domain $Y$. For the class of bounded, continuous functions appearing in~\eqref{eq:g_integral2} and~\eqref{eq:g_quadrature} as integrands, $s_x$ is bounded over $x \in \Omega$, and we can conclude that $g_{\epsilon,\bm N}$ with $\bm N = (N, \ldots, N)$  converges uniformly to $g_\epsilon$ as $N\to\infty$.

Two important advantages of the Monte Carlo approach are its simplicity and flexibility---other than boundedness of the standard deviations $s_x$ we had to make essentially no explicit assumptions on the properties of $\varphi_x$ and/or integration domains $Y$. At the same time, the standard Monte Carlo scheme does not take into account any additional regularity properties that the integrands may possess that could lead to faster rates of convergence than $N^{-1/2}$. The scheme can also be inefficient in situations where $Y$ is not well-bounded by a rectangle (meaning, $\lvert R\rvert / \lvert Y\rvert$ is large), causing many trial quadrature nodes $\tilde y_j$ drawn from the Lebesgue measure on $R$ to be discarded.

\subsection{Triangulation method}
\label{sec:triangulation}

To overcome the aforementioned shortcomings of the Monte Carlo approach, we use a deterministic quadrature scheme that is based on a triangulation of $\bar\Omega$ and the subdomains $S_\ell$. This approach sacrifices some generality, as it assumes that the integration domains are polygonal and the integrands $\varphi_x$ are smooth, in exchange of potentially obtaining faster convergence rates than~\eqref{eq:mc_estimate}.

We use a triangulation algorithm developed by Shewchuk \cite{Shewchuk2002}, which is a modification of Ruppert's algorithm for two-dimensional quality mesh generation \cite{Ruppert1995}. Through this approach, we generate triangulations of $\bar\Omega$ and $S_\ell$ that satisfy constraints on the minimum angle and maximum triangle area; see \cref{fig:poly_triangulation} for a schematic of this method. We use an implementation of Shewchuk's algorithm provided in the Python library \texttt{triangle}. Additional details on numerical implementation are included in the online repository accompanying the paper.

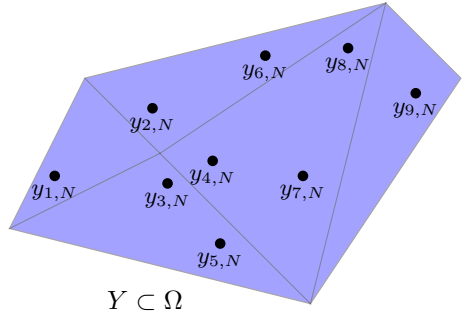
\begin{figure}
    \centering
\begin{tikzpicture}

\filldraw[fill=blue, draw=black, opacity=0.2]
  (0,0) -- (4,-1) -- (6,2) -- (5,3)  -- (1,2)-- cycle;

\filldraw[fill=blue, draw=black, opacity=0.2]
  (0,0) -- (4,-1) -- (2,1) -- cycle;

\filldraw[fill=blue, draw=black, opacity=0.2]
 (4,-1)-- (5,3) -- (2,1) -- cycle;

\filldraw[fill=blue, draw=black, opacity=0.2]
 (5,3) -- (1,2) -- (2,1) -- cycle;

 \filldraw[fill=blue, draw=black, opacity=0.2]
 (1,2) -- (0,0) -- (2,1) -- cycle;

 \filldraw[fill=blue, draw=black, opacity=0.2]
 (5,3) -- (4,-1) -- (6,2) -- cycle;

 \fill (0.6,0.7) circle (2pt) node[below] {\small $y_{1,N}$};
 %\fill (1.2,0.8) circle (2pt) node[below] {\small $y_{2,N}$};
 %\fill (1.4,-0.2) circle (2pt) node[below] {\small $y_{3,N}$};
 \fill (1.9,1.6) circle (2pt) node[below] {\small $y_{2,N}$};
 \fill (2.1,0.6) circle (2pt) node[below] {\small $y_{3,N}$};
 \fill (2.7,0.9) circle (2pt) node[below] {\small $y_{4,N}$};
 \fill (2.8,-0.2) circle (2pt) node[below] {\small $y_{5,N}$};
  \fill (3.4,2.3) circle (2pt) node[below] {\small $y_{6,N}$};
 \fill (3.9,0.7) circle (2pt) node[below] {\small $y_{7,N}$};
 %\fill (4.1,-0.2) circle (2pt) node[below] {\small $y_{10,N}$};
  \fill (4.5,2.4) circle (2pt) node[below] {\small $y_{8,N}$};
  %\fill (5.05,0.8) circle (2pt) node[below] {\small $y_{12,N}$};
  \fill (5.4,1.8) circle (2pt) node[below] {\small $y_{9,N}$};

  \node[below] at (1.8,-0.7)  {$Y \subset \Omega$};

\end{tikzpicture}
\caption{Schematic of a triangulation of a single polygonal domain $Y$ together with associated quadrature nodes $y_{j,N}\in Y$ (each coming with a corresponding weight $w_{j,N}$).}
%\dg{The figure needs annotations, e.g., of the domain $Y$ and representative quadrature nodes $y_{j,N}$.}
    \label{fig:poly_triangulation}
\end{figure}

Next, we examine the numerical quadrature error over a triangulated domain for piecewise-constant functions $f_\ell : S_\ell \to \mathbb R$ and the radially symmetric function $h_\epsilon: \mathbb R^2 \to \mathbb R$ induced from the Gaussian RBF in~\eqref{eq:h_gaussian}. That is, on each polygonal domain $Y$ (where $Y$ is either $\Omega$ or one of the subdomains $S_\ell$) we are interested in approximating integrals of the form $I[\varphi_x] = \int_Y \varphi_x(y) \, dy$ where $\varphi_x = \eta_{\epsilon, x}$ is the kernel section induced from
\[
    h_\epsilon(y) = h_\epsilon(r)=\frac{1}{\epsilon \sqrt{2 \pi}} e^{-\frac{r^2}{2 \epsilon^2}}, \quad r = \lVert y\rVert,
\]
and $x$ is a point in $\Omega$.

Let $\mathcal T = \{ T_j \}_{j=1}^N$ be a triangulation of $Y$. For every triangle $T_j \in \mathcal T$ and $a, b \in T_j$ we have $|\varphi_x(a)-\varphi_x(b)| \leq \lVert h'_\epsilon\rVert_\infty \diam T_j$, where $h'_\epsilon : (0,\infty) \to \mathbb R$ is the radial derivative of $h_\epsilon$.

Let us obtain an upper bound for $\lVert h'_\epsilon\rVert_\infty $. We have,
$$
h'_\epsilon(r)=-\frac{r}{\epsilon^2} h_\epsilon(r), \quad h_\epsilon''(r) =\left(\frac{r^2}{\epsilon^2}-1\right) \frac{h_\epsilon(r)}{\epsilon^2}.
$$
The extrema of $h'_\epsilon(r)$ are at $r= \pm \epsilon$, and thus
\begin{equation}
    \label{eq:h_deriv_bound}
    \lVert h'_\epsilon\rVert_\infty = \frac{1}{\epsilon} h_\epsilon(\epsilon)=\frac{e^{-\frac{1}{2}}}{\epsilon^2 \sqrt{2\pi}}.
\end{equation}
If we estimate the integral $I_j := \int_{T_j} \varphi_x(y) \, dy = \int_{T_j} h_\epsilon(y-x) \, dy$ by
\begin{equation}
    \label{eq:quad_piecewise_const}
    \hat I_j := h_\epsilon(y_j - x) \area T_j,
\end{equation}
where $y_j$ is any point in $T_j$, \eqref{eq:h_deriv_bound} gives the error bound
$$
E_j = \lvert I_j - \hat I_j\rvert \leq \delta_{\epsilon} \diam T_j \area T_j.
$$
Now, for every polygonal domain $Y$, there exists an angle $\alpha \in (0, \pi)$, a constant $A > 0$, and a family of triangulations $\mathcal T_1, \mathcal T_2, \ldots$ with $\mathcal T_N= \{ T_{N,j} \}_{j=1}^N$ such that (i) the angles of every triangle in $\bigcup_{N\in \mathbb N} \mathcal T_N $ are bounded below by $\alpha$; and (ii) the maximum area of the triangles in $\mathcal T_N$ are bounded above by $A / N$. By the triangle area formula, this means that there exists a constant $c_\alpha$ such that for all $N \in \mathbb N$ and $j \in \{1, \ldots, N \}$,
$$
\area T_{N,j} \geq c_\alpha (\diam T_{N,j})^2.
$$
This leads to the following bound for the total error $\mathcal E_N := \lvert I[\varphi_x] - I_N[\varphi_x]\rvert$:
\begin{align}
    \nonumber \mathcal E_N & \leq \sum_j E_j \leq \sum_{j=1}^N \delta_{\epsilon} \area \left(T_j\right) \diam \left(T_{N,j}\right) \leq \sum_{j=1}^N \delta_{\epsilon}c_\alpha^{-\frac{1}{2}} \area \left(T_{N,j}\right)^{\frac{3}{2}}\\
    \label{eq:quad_error} &  \leq \delta_\epsilon c_\alpha^{-1/2} A^{3/2} N^{-1/2}.
\end{align}
Thus, with such a triangulation, as $N\to\infty$ the error associated with the quadrature based on~\eqref{eq:quad_piecewise_const} converges to 0 at a rate $O(N^{-1/2})$ for piecewise-constant functions.

Let $\alpha_* \in (0, \pi)$ be the minimal edge angle over all vertices of the boundary of $Y$. Given an area parameter $a>0$, the algorithm of \cite{Shewchuk2002} produces a triangulation $\mathcal T_{N_a}$ of $Y$ with $N_a$ elements such that (i) the minimal angle of every triangle in $\mathcal T_{N_a}$ is at least $\alpha_*$; and (ii) the maximal area of the triangles in $\mathcal T_N$ is no greater than $a$. In addition, under refinement of the triangulation, $ a \mapsto a / k $ for $k \in \mathbb N$, the number of triangles increases linearly with $k$.  This allows the convergence rate in~\eqref{eq:quad_error} to be attained on a sequence $\mathcal T_{N_{a_1}}, \mathcal T_{N_{a_2}}, \ldots$ with $a_k = a /k$. Faster rates of convergence can be obtained using higher-order quadrature schemes at the expense of a larger prefactor $\delta_\epsilon$ associated with higher-order derivatives of $h_\epsilon$. In this paper, we do not explore these options further as we found that the simple quadrature formula~\eqref{eq:quad_piecewise_const} gave satisfactory numerical results.

\section{Applications to sea ice DEM data}
\label{sec:experiments}

We consider a time-evolving, finite collection of ice floes in a thick, bounded, open domain $\Omega \subset \mathbb R^2$. At any given time $t \in \mathbb R$, each floe is labeled by an index $\ell$ in a (finite) index set $\Lambda_t$ and occupies a bounded, closed subset $S_{t,\ell} \subseteq \bar\Omega$. In addition, every ice floe has associated physical properties which are represented by continuous functions defined on $S_{t,\ell}$.

In what follows, we will be interested in mass density, velocity, and stress tensor fields, denoted as  $m_{t,\ell} : S_{t,\ell} \to \mathbb R$, $\bm v_{t,\ell} : S_{t,\ell} \to \mathbb R^2$, and $\bm \sigma_{t, \ell} : S_{t,\ell} \to \mathbb R^{2 \times 2}$, respectively. We will denote a generic such function by $\tilde f_{t,\ell} : S_{t,\ell} \to \mathbb R^d$. At every time $t$, the floe-local functions $\tilde f_{t,\ell}$ collectively define a field $f_t: \Omega \to \mathbb R^d$ such that $f_t(x) = \sum_{l\in\Lambda_t} f_{t,\ell}$
where $f_{t,\ell}: \Omega \to \mathbb R^d$ is an extension of $\tilde f_{t,\ell}$ on the domain $\Omega$ defined as
\begin{equation}
    f_{t,\ell}(x) =
    \begin{cases}
        \tilde f_{t,\ell}(x), & x \in S_{t,\ell},\\
        0, & \text{otherwise}.
    \end{cases}
    \label{eq:f_ext}
\end{equation}
Note that $f_t$ is a well-defined, piecewise-continuous function on $\Omega$ since $\Lambda_t$ is a finite set and the floe-local functions $\tilde f_{t,\ell}$ are continuous. In particular, $f_t$ is of the form~\eqref{eq:f_orig} with $f_\ell$ in that equation set to the components of the $\mathbb R^d$-valued functions $f_{t,\ell}$.  It should also be noted that the definition~\eqref{eq:f_orig} allows for potential overlaps of the domains $S_{t,\ell}$, which may occur during floe--floe collisions.

Within this setup, we build a family of smooth approximations $g_{t,\epsilon, \bm N_t} : \Omega \to \mathbb R^d$ to $f_t$ obtained by componentwise application of~\eqref{eq:g_integral2} followed by numerical approximation via triangulation method from \cref{sec:triangulation}. Specifically, for a kernel lengthscale parameter $\epsilon$ and quadrature parameterized by the multi-index $\bm N_t$, we use~\eqref{eq:g_quadrature} to obtain kernel-smoothed approximations of $m_t$, $\bm v_t$, and $\bm \sigma_t$, denoted by $m_{t,\epsilon, \bm N_t}: \Omega \to \mathbb R$, $\bm v_{t,\epsilon, \bm N_t}: \Omega \to \mathbb R^2$, and $\bm \sigma_{t,\epsilon,\bm N_t}: \Omega \to \mathbb R^{2\times 2}$, respectively. Throughout, we employ the Gaussian RBF \eqref{eq:h_gaussian} as the kernel shape function $h$. The Gaussian RBF satisfies all of \cref{prop:K1,prop:K2,prop:K3,prop:K4,prop:K5,prop:K6} and the smoothed fields $m_{t,\epsilon,\bm N_t}$, $\bm v_{t,\epsilon,\bm N_t}$, and $\bm \sigma_{t,\epsilon,\bm N_t}$ are all infinitely differentiable.

\subsection{Ice floe variables}
\label{sec:dem_output}

The mass density, velocity field, and stress tensor field are employed in virtually all DEMs and continuum models of sea ice. In what follows, we describe their properties in a scenario where each ice floe is a spatially homogeneous object undergoing rigid body motion and experiencing point contact forces resulting from collisions with other ice floes and/or the domain boundaries. This scenario reflects how sea ice is modeled in the SubZero DEM \cite{ManucharyanMontemuro22} that we employ in our numerical experiments. Other modeling assumptions and/or physical variables can be handled by suitable modifications of the schemes described here.

\paragraph*{Mass density field}

Under our assumption of spatially homogeneous ice floes, the mass density function $m_{t,\ell}$ on each ice floe is constant; that is, we have $m_{t,\ell}(x) = \mu_{t,\ell}$ for all $x \in S_{t,\ell}$ where $\mu_{t,\ell}$. In the SubZero simulations analyzed in this paper, each ice floe has a fixed thickness $a_{t,\ell}$, and we have $\mu_{t,\ell} = \rho a_{t,\ell}$ where $\rho$ is the (constant) density of sea ice. Note that $a_{t,\ell}$ may undergo changes over time resulting from floe processes such as fracture and welding.

\paragraph*{Velocity Field}
Let $\xi_{t,\ell} \in \mathbb R^2$ be the center of mass for the floe $S_{t,\ell}$ and $\bm u_{t,\ell} \in \mathbb R^2$ the corresponding center-of-mass velocity. Under our assumption of rigid-body motion, the ice velocity $\bm v_{t,\ell}(x)$ at a point $x \in S_{t,\ell}$ is the sum of contributions from the center of mass velocity and rotational velocity; i.e., $\bm v_{t,\ell}(x) = \bm u_{t,\ell} + \omega_{t,\ell} \bm r^\perp$,  where $\omega_{t,\ell}$ is the angular velocity (taken positive in the counter-clockwise sense), $\bm r = x - \xi_{t,\ell} \in \mathbb R^2$ is the displacement vector of $x$ from the center of mass $\xi_{t,\ell}$, and $\bm r^\perp \in \mathbb R^2$ is the perpendicular vector to $\bm r $ in the counter-clockwise sense. See \cref{fig:rot_polygon} for a schematic illustration.

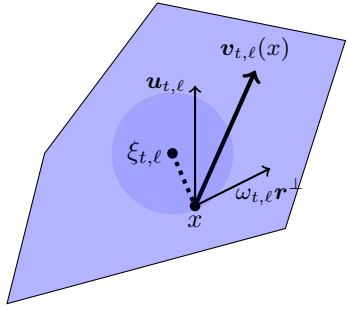
\begin{figure}
    \centering
       \begin{tikzpicture}
          \draw[fill=blue!30] (-0.5,0) -- (1,2) -- (3.5,1.5) -- (2.7,-1) -- (-1,-2) -- cycle;
          \draw[->,thick] (1.5,-0.7) -- (2.5,-0.2);
          \node[below] at (2.5,-0.2)  {\small $\omega_{t,\ell} \bm r^\perp$};

          \draw[dotted, ultra thick] (1.2,0) -- (1.5,-0.7);

          \draw[->, thick, black] (1.5,-0.7) -- (1.5,0.9);
          \node[left] at (1.5,0.9)  {\small $\bm u_{t,\ell}$};

            \draw[->, ultra thick, black] (1.5,-0.7) -- (2.3,1.1);
          \node[above] at (2.3,1.1)  {\small $\bm v_{t,\ell}(x)$};

          \fill (1.2,0) circle (2pt) node[left] {$\xi_{t,\ell}$};
          \fill (1.5,-0.7) circle (2pt) node[below] {$x$};
          \draw[fill=blue, opacity=0.1] (1.2,0) circle (0.8cm);
        \end{tikzpicture}
        \caption{Illustration of the center-of-mass velocity $\bm u_{t,\ell}$, rotational velocity $\omega_{t,\ell} \bm r^\perp$, and total velocity $\bm v_{t,\ell}$ at a point $x \in S_{t,\ell}$.}
    \label{fig:rot_polygon}
\end{figure}

\paragraph*{Stress tensor field}

The stress tensor field $\bm \sigma_{t,\ell}$ on each ice floe is a constant resulting from contact forces it experiences at any given time. Assuming that floe $S_{t,\ell}$ is subjected to contact forces $\bm f_{t,\ell,1}, \ldots, \bm f_{t,\ell,M} \in \mathbb R^2$ acting at points $z_{t,\ell,1}, \ldots, z_{t,\ell,M}$ on its boundary, the resulting stress tensor is given by
\begin{displaymath}
    \bm \sigma_{t,\ell}(x) = \sum_{i=1}^M (\bm r_{t,\ell,i} \bm f_{t,\ell,i}^\top + \bm f_{t,\ell,i}^\top \bm r_{t,\ell,i}),
\end{displaymath}
where $\bm r_{t,\ell,i} = z_{t,\ell,i} - \xi_{t,\ell}$ is the displacement vector of contact point $z_{t,\ell,i}$ from the center of mass $y_{t,l}$ and we treat $\bm f_{t,\ell,i}$ and $\bm r_{t,\ell,i}$ as column vectors.

\subsection{SubZero simulations}

We analyze two ice floe simulations performed using SubZero: (i) Nares Strait simulation; and (ii) lateral compression simulation. Results from these experiments are described in \cref{sec:nares,sec:compression}, respectively.

In each case, the domain $\Omega$ is a subset of a square $\mathcal S \subset \mathbb R^2$ centered at the origin, whose sides have length $140 $ km. In the Nares Strait case, we build $\Omega$ by removing from $\mathcal S$ points corresponding to the Canadian and Greenland land masses bounding the strait; see \cref{fig:nares}. In the lateral compression case, $\Omega$ is the entire square $\mathcal S$.

Each SubZero run generates a time series of ice floes with domains $ \{ S_{t,\ell} \subseteq \Omega \}_{\ell \in \Lambda_t}$ and associated thickness $a_{t,\ell}$, center-of-mass velocity $\bm u_{t, \ell}$, angular velocity $\omega_{t,\ell}$, and contact forces $\bm f_{t,\ell,i}$. We sample these time series at a fixed time interval $\Delta t = Q t_+$, where $Q$ is a positive integer  and $t_+$ the simulation timestep. In our simulations we have $t_+ = 5$ s and $Q = 100$, leading to a sampling interval $\Delta t = 500$ s.

Using the sampled floe data, we compute associated mass, velocity, and stress tensor fields, $m_{t,\ell}$, $\bm v_{t,\ell}$, and $\bm \sigma_{t,\ell}$, respectively, as described in \cref{sec:dem_output}. We then assemble these variables into piecewise-continuous mass density, velocity, and stress tensor fields, $m_t$, $\bm v_t$, and $\bm \sigma_t$, respectively, defined on $\Omega$. For the computation of the stress tensor, we aggregate the contact forces $\bm f_{t,\ell,i}$ on each ice floe over 100 timesteps, appropriately taking into account the floe's fracture history. We smooth these fields using the Markov kernel \eqref{eq:p_kernel} derived from the Gaussian RBF \eqref{eq:h_gaussian} with lengthscale parameter $\epsilon = 1.7$ km.

We numerically approximate the corresponding integral operator $P_\epsilon$ using the triangulation-based quadrature scheme described in \cref{sec:integration_scheme} with maximum triangle area of 0.5 $\text{km}^2$. In our analyzed simulations, the number of ice floes at any time $t$ varies (e.g., due to fracture) between 200 and 1300 and we take one quadrature point from each triangle. Applying the discretized integral operator leads to smooth density, velocity, and stress fields, which we denote here as $\hat m_t \equiv m_{t,\epsilon,\bm N_t} $, $\hat{\bm v}_t \equiv \bm v_{t,\ell, \bm N_t}$, and $\hat{\bm \sigma}_t \equiv \bm \sigma_{t,\ell, \bm N_t}$, respectively. We evaluate the smooth fields on a collection of points $ \{ \bm x_i \} \subset \Omega$ given by the subset of the points in a  $800 \times 800$ uniform grid of the bounding square $\mathcal S$ that lie in $\Omega$. The corresponding grid resolution is $\text{$140/800$ km} = \text{175 m}$, but note that $\hat m_t$, $\hat{\bm v}_t$, and $\hat{\bm \sigma}_t$ are everywhere-defined smooth fields in $\Omega$.

\subsection{Nares Strait simulation}
\label{sec:nares}

The Nares Strait is a narrow channel between Ellesmere Island in the Canadian Archipelago and Greenland that connects the Lincoln Sea of the Arctic Ocean to Baffin Bay in the Atlantic Ocean. Having a length of approximately 500 km and a minimum width of approximately 35 km, it plays an important role in transport between the Arctic and Atlantic Oceans, exhibiting complex sea ice dynamical phenomena such as jamming and arching that have been the focus of several modeling studies with both continuum models and DEMs \cite{ManucharyanMontemuro22,WestEtAl22}. As our first experiment, we examine an idealized simulation of sea ice motion through a domain with a Nares Strait topography, visualized in \cref{fig:nares}.

\begin{figure}
    \centering
      \includegraphics[width=0.7\linewidth]{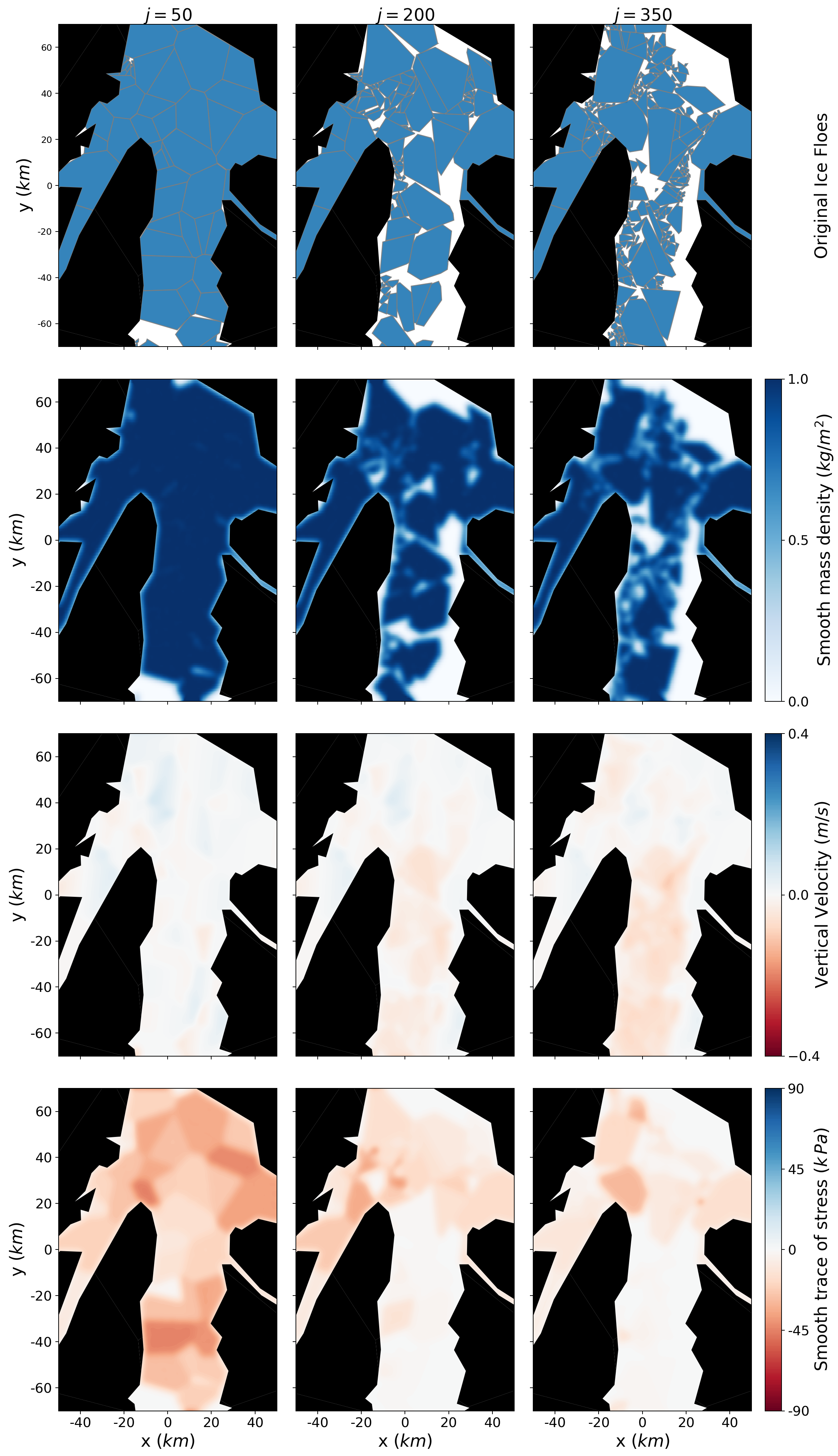}
      \caption{Nares Strait simulation. The top row shows snapshots of the ice floes from the DEM at simulation times $t_j =j \, \Delta t$ with $j= 50$ (left), 200 (center), and 350 (right). The second, third, and fourth rows from the top show snapshots of kernel-smoothed mass density $\hat m_t$, meridional velocity ($y$ component of $\hat{\bm v}_t$), and trace of the stress field $\tr \hat{\bm \sigma_t}$, respectively, at the times $t_j$.}
      \label{fig:nares}
\end{figure}

Ice floes of various shapes and sizes are initially placed in the domain $\Omega$ in a fully packed configuration. The floes are forced southward into the strait by quadratic drag induced from a constant southerly wind of velocity 10 $\text{m}\, \text{s}^{-1}$. In the ensuing simulation, (depicted in the left-hand column of \cref{fig:nares} at representative time instances with time increasing downwards) the floes experience a buildup of stress due to jamming. Eventually, floes begin to fracture, leading to a release of stress and decrease of average floe size that allows ice transport through the strait. In this simulation, floes are unbonded. Contact forces with the domain boundary are determined by numerically treating the coastlines as ice floes that are immovable but otherwise interact with other floes through unmodified contact forces.

The second, third, and fourth columns of \cref{fig:nares} from the left display scatterplots of the mass density field $\hat m_t$, the meridional (North--South) component of the velocity field $\hat{\bm v}_t$, and the trace of the tensor field $\hat{\bm \sigma}_t$. We use the latter to diagnose compressive stresses, which are associated with large negative values of $\tr\hat{\bm \sigma}_t$. Examining the top row in \cref{fig:nares}, we see that early on in the simulation the floes exhibit negligible meridional motion and are subjected to large compressive stresses induced by jamming and wind drag. At later times (second row) we see considerable reduction of stress in the southern half of the domain, as floes located there begin to fracture and move southward towards the exit of the strait. At late times (bottom row of \cref{fig:nares}), we see additional buildup of southward velocity in the southern half of the domain as well as residual compressive stresses in the northern half due to jamming between unfractured floes.

\subsection{Lateral compression simulation}
\label{sec:compression}

In the second experiment, a lateral compression (i.e., compression directed along the $y$-axis) is applied to a square piece of ice of size $\text{100 km} \times \text{100 km}$  located in the center of the $\text{140 km} \times \text{140 km}$  domain $\Omega$. The square ice piece consists of a collection of 200 ice floes of various shapes, arranged in a densely packed configuration. Unlike the Nares Strait case, ice floes in this example are bonded.  Numerically, the bonds are implemented as small-sized ice floes of a ``bow-tie'' shape, which are inserted randomly in the boundaries between floes. The north/south boundaries move at a constant prescribed velocity of $0.1 \, \text{m}\, \text{s}^{-1}$ toward the center. The atmospheric and oceanic stresses are set to zero in this simulation. As stress builds up due to the externally imposed lateral stress, the bow-tie floes can fracture under spatially-uniform fracture criteria, allowing the floes to move in directions transverse to the applied pressure. The resulting dynamical evolution exhibits fracture patterns with long-range correlations resembling LKFs. See the top row of \cref{fig:vertical} for a visualization of the simulation at representative time instances.

\begin{figure}
    \centering
      \includegraphics[width=0.8\linewidth]{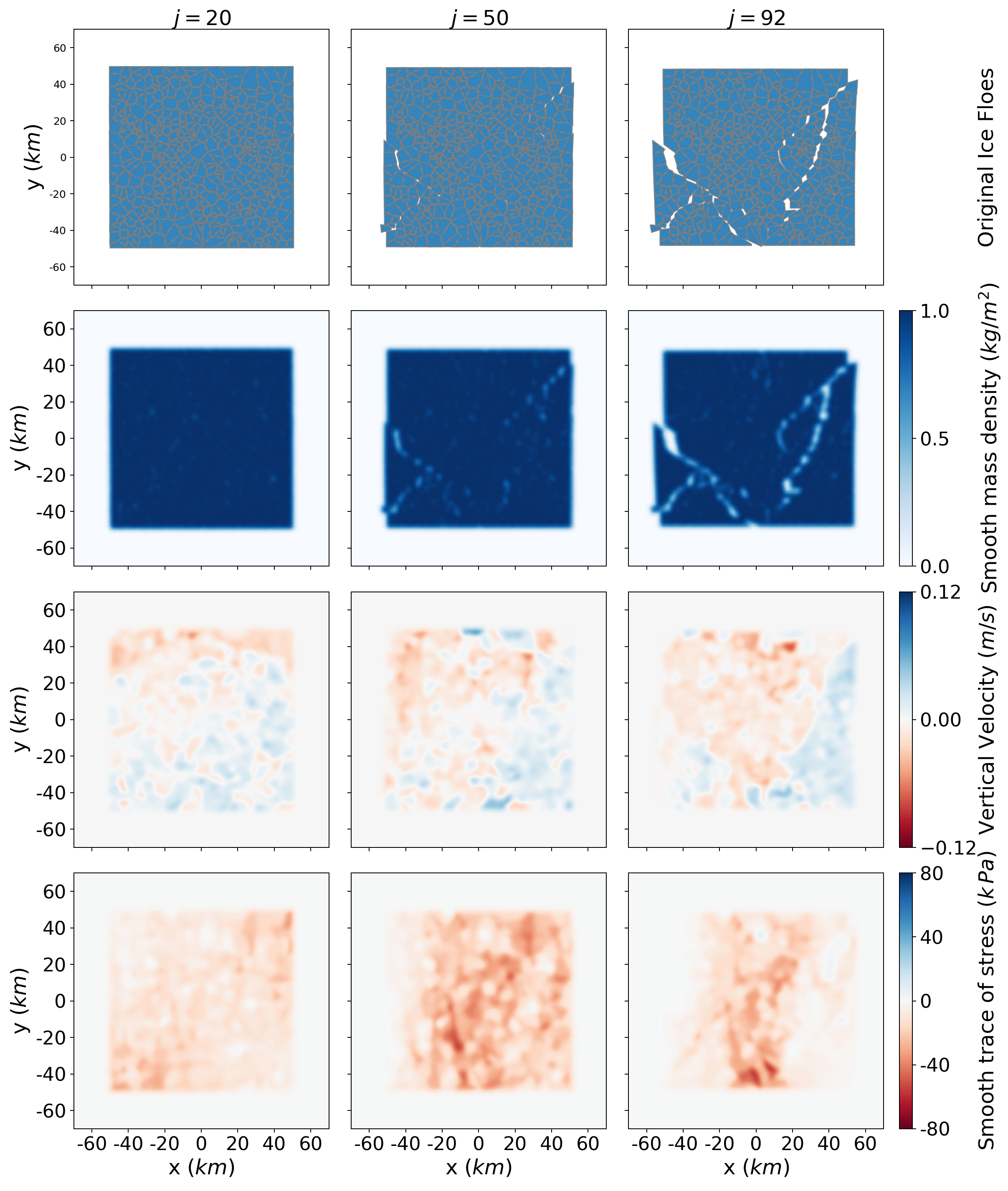}
      \caption{As in \cref{fig:nares}, but for the lateral compression experiment. The time instances $t_j = j \, \Delta t$ shown are for $j=20$ (left), 50 (center), and 92 (right).}
      \label{fig:vertical}
\end{figure}

The second, third, and fourth rows of \cref{fig:vertical} from the top show scatterplots of the kernel-smoothed mass density $\hat m_t$, the $y$ component of velocity $\hat{\bm v}_t$, and trace of the stress tensor $\hat{\bm \sigma}_t$ in a similar manner as \cref{fig:nares}. Early on in the simulation (left column), the stress field exhibits a fairly diffuse compressive pattern in space as the floes are tightly packed and bonds as unbroken. As bonds start to break (center and right columns), we see that stress becomes increasingly concentrated in a wedge-shaped region of high ice concentration that connects the lateral domain boundaries where stress is applied. Meanwhile, chunks of ice adjacent to the central wedge-shape region move sideways, leading to a noticeable pattern in the velocity field component. As with the Nares Strait simulation, the smoothed stress tensor field $\hat{\bm \sigma}_t$ reveals stress patterns that are not readily noticeable in the raw DEM contact forces.

\section{Conclusions}
\label{sec:conclusions}

In this paper, we studied the convergence properties of one-parameter families of integral operators with Markovian kernels constructed from normalization of convolution operators associated with radial kernel functions. Using the theory of Hardy--Littlewood maximal operators in conjunction with a thickness condition that bounds the Lebesgue density of the spatial domain $\Omega$ from below, we proved pointwise and $L^p$ norm convergence results for continuous approximations of $ f \in L^p(\Omega)$, $p \in (1, \infty)$, by $P_\epsilon f$ for Markov operators $P_\epsilon$ obtained by left normalization of integrable radial kernels with lengthscale parameter $\epsilon$ (\cref{thm:main}). We also showed $L^p$ norm convergence for approximations $\tilde P_\epsilon f$ obtained from Markov operators with bistochastic kernels (\cref{thm:bistoch}) that preserve integrals of functions in addition to preserving constant functions. The paper also studied properties of thick sets from the perspective of metric measure spaces, where it was shown that thickness implies a locally doubling condition on $\Omega$ but is in general independent from (global) doubling (\cref{prop:doubling,prop:not_doubling}).

Numerically, our approach can be implemented using a variety of quadrature schemes, including Monte Carlo and mesh-based schemes. Here, we have focused on a deterministic quadrature scheme based on triangulation of polygonal domains. Using this approach, we built smooth approximations of piecewise-continuous mass density, velocity, and stress tensor fields from discrete element models (DEMs) of sea ice dynamics. We used the smoothed fields to characterize stress patterns associated with jamming and fracture of wind-forced floes passing through the Nares Strait and the formation of idealized linear kinematic features resulting from lateral compression. These results demonstrate the utility of kernel smoothing for postprocessing and retrieval of physically meaningful information from DEM simulations.

Throughout this work, the Lebesgue measure on $\mathbb R^n$ played a central role by providing the background measure for Markov normalization. A possible avenue of future work would be to study the convergence properties of Markov operators built from measures other than Lebesgue, such as invariant measures of dynamical systems. From the point of view of applications, it would be fruitful to explore use-cases of the reconstructed smooth fields to infer macro-scale constitutive laws emerging from DEM simulations, along with associated rheology parameterization schemes.

\appendix

\section{Proofs of auxiliary results}
\label{app:proofs}

\begin{proof}[Proof of \cref{lem:conv_contin}]
    For every $q \in [1,\infty)$, the group of translation operators $\{ T_x : L(\mathbb R^n) \to L(\mathbb R^n) \}_{x \in \mathbb R^n}$ extends to a strongly continuous group on $L^q(\mathbb R^n)$ that acts by isometries; that is, the mapping $ x \mapsto T_x h \in L^q(\mathbb R^n) $ is norm-continuous at every $x \in \mathbb R^n$ for every $ h \in L^q(\mathbb R^n)$ and $\lVert T_x h \rVert_{L^q(\mathbb R^n)} = 1$. Since $L^p(\mathbb R^n) \simeq L^{q*}(\mathbb R^n)$, $\varphi : h \mapsto \int_{\mathbb R^n} f(y) h(y)\, dy$ is a bounded functional on $L^q(\mathbb R^n)$. Moreover, for every $x \in \mathbb R^n$ we have $ f * h(x) = \varphi((T_x h)^\checkmark)$. The continuity of $h * f$ follows from the facts that $\varphi$ and $^\checkmark$ are continuous maps on $L^q(\mathbb R^n)$ and $\{ T_x : L^q(\mathbb R^n) \to L^q(\mathbb R^n)\}_{x \in \mathbb R^n}$ is strongly continuous. The bound $ \lVert h * f \rVert_\infty \leq \lVert h \rVert_{L^q(\mathbb R^n)} \lVert f\rVert_{L^p(\mathbb R^n)}$ follows from H\"older's inequality and the fact that $T_x$ acts on $L^q(\mathbb R^n)$ as an isometry.
\end{proof}

\begin{proof}[Proof of \cref{lem:sobolev}]

    By a change of variables, it is enough to prove the claim for $\epsilon = 1$.

    Starting from the $r=1$ case, let $e_1, \ldots, e_n \in \mathbb R^n$ be the standard basis vectors of $\mathbb R^n$ and $D_j : E_j \to \mathbb L^2(\mathbb R^n)$ the generators of the (unitary) translation group on $L^2(\mathbb R^n)$ along $e_j$. That is, $D_j$ is a skew-adjoint operator defined on a dense subspace $E_j \subset L^2(\mathbb R^n)$ by the $L^2(\mathbb R^n)$-norm limit $ D_j h = \lim_{s\to 0} (T_{se_j} h - h) / s$. Let $\hat E_j = \{ \hat h \in L^2(\mathbb R^n) : \int_{\mathbb R^n} \omega_j^2 \lvert \hat h(\omega)\rvert^2 \, d\omega < \infty \}$. The generators $D_j$ map under the unitary Fourier transform $\mathcal F: L^2(\mathbb R^n) \to L^2(\mathbb R^n)$ to the multiplication operators $\hat D_j : \hat E_j \to L^2(\mathbb R^n)$ where $\hat D_j = \mathcal F^* D_j \mathcal F$ and $\hat D_j \hat h(\omega) = i \omega_j \hat h(\omega)$. For $h \in H^1(\mathbb R^n)$ we have that $\omega \mapsto i\omega_j \hat h(\omega) \in L^2(\mathbb R^n)$ since $\omega \mapsto (1 + \lVert \omega\rVert^2)^{1/2} \hat h(\omega) \in L^2(\mathbb R^n)$ and $\lvert \omega_j \rvert \leq C(1 + \lVert \omega\rVert^2)^{1/2}$ for a constant $C$. As a result, $h$ lies in $E_j$. By properties of one-parameter unitary groups (e.g., \cite[Proposition~6.5]{Schmudgen12}), the function $u(t) = T_{te_j} h$ lies in $C^1(\mathbb R, L^2(\mathbb R^n))$; that is, $u'(t) = \lim_{s\to 0}(T_{(t+s)e_j} h - T_{te_j} h) /s$ is well-defined and continuous at every $t \in \mathbb R$. As a result, for every $f \in L^2(\mathbb R^n)$ and $x \in \mathbb R^n$, the function $t \mapsto \varphi((T_{t e_j} h)^\checkmark)$ with $\varphi h = \int_{\mathbb R^n} f(y) h(y) \, dy$ lies in $C^1(\mathbb R)$. Writing $x = \sum_{j=1}^n x_j e_j $ with $x_j \in \mathbb R$, we conclude that  $x \mapsto \varphi((T_x h)^\checkmark) = \varphi((T_{x_1 e_1} \circ \cdots \circ T_{x_n e_n} h)^\checkmark)$ lies in $C^1(\mathbb R^n)$.

    Next, for $ r\geq 1$, we have $(1 + \lVert \omega\rVert^2)^{r/2} \geq C \lvert \omega_j\rvert (1 + \lVert \omega\rVert^2)^{(r-1)/2}$ which implies that $D_j h \in H^{r-1}(\mathbb R^n)$ whenever $h \in H^r(\mathbb R^n)$. Using this fact, we show by induction that $u(t) = T_{te_j} h$ lies in $C^r(\mathbb R, L^2(\mathbb R^n))$; recursively, this means that $u'(t) = \lim_{s\to 0}(T_{(t+s)e_j} h - T_{te_j} h) /s$ lies in $C^{r-1}(\mathbb R, L^2(\mathbb R^n))$. We have already proved the base case, $r=1$, so suppose that for $r-1 \in \mathbb N_0$ it holds that for every $g \in H^{r-1}(\mathbb R^n)$, $w(t) = T_{te_j} g$ lies in $C^{r-1}(\mathbb R, L^2(\mathbb R^n))$. In that case, for $h \in H^r(\mathbb R^n)$ and $u(t) = T_{te_j} h$ we have $u'(t) = D_j T_{te_j} h = T_{te_j} D_j h = T_{te_j} g$ with $g = D_j h \in H^{r-1}(\mathbb R^n)$. We therefore deduce that $u' \in C^{r-1}(\mathbb R, L^2(\mathbb R^n))$ and thus that $u \in C^r(\mathbb R, L^2(\mathbb R^n))$. With this result and $f \in L^2(\mathbb R^n)$, $\varphi \in L^{2*}(\mathbb R^n)$, and $x=\sum_{j=1}^n x_j e_j$ as above, it follows that $t \mapsto \varphi((T_{t e_j} h)^\checkmark)$ lies in $C^r(\mathbb R)$ and $x \mapsto \varphi((T_x h)^\checkmark) = \varphi((T_{x_1 e_1} \circ \cdots \circ T_{x_n e_n} h)^\checkmark)$ lies in $C^r(\mathbb R^n, \mathbb R)$.
\end{proof}

\begin{proof}[Proof of \cref{lem:mult_conv_s}]
    Let $A$ be an upper bound for $\lVert f\rVert_{L^\infty(\Omega)}$ and $\lVert f_\epsilon\rVert_{L^\infty(\Omega)}$ and fix $g \in L^p(\Omega)$. For every $\delta >0$, there exists $\tilde g \in C_c(\Omega)$ such that $\lVert g - \tilde g\rVert_{L^p(\Omega)} < \delta$. We have
    \begin{displaymath}
        \begin{aligned}
            \lVert M_{f_\epsilon} g - M_f g\rVert_{L^p(\Omega)}
            &= \lVert (f_\epsilon - f) g\rVert_{L^p(\Omega)} \\
            & \leq \lVert (f_\epsilon - f)(g - \tilde g)\rVert_{L^p(\Omega)} + \lVert (f_\epsilon - f) \tilde g\rVert_{L^p(\Omega)}\\
            & \leq \lVert f_\epsilon - f\rVert_{L^p(\Omega)} \delta + \lVert f_\epsilon - f\rVert_{L^p(\Omega)} \lVert \tilde g\rVert_\infty\\
            & \leq 2 A \delta + \lVert f_\epsilon - f\rVert_{L^p(\Omega)} \lVert \tilde g\rVert_\infty.
        \end{aligned}
    \end{displaymath}
    The claim follows from the fact that $\delta$ was arbitrary and $\lim_{\epsilon\to 0} \lVert  f_\epsilon - f\rVert_{L^p(\Omega)} = 0$.
\end{proof}

\bibliographystyle{siamplain}
\bibliography{refs,bibliography_dg}

\end{document}